\documentclass[a4,10pt]{article}
\usepackage[left=4cm,right=4cm,top=3cm,bottom=3cm]{geometry}

\usepackage{amssymb,amsmath, amsthm, amsfonts}
\usepackage{mathrsfs}
\usepackage{mathabx}
\usepackage{arydshln}
\newtheorem{theorem}{Theorem}[section]
\newtheorem{lemma}[theorem]{Lemma}
\newtheorem{corollary}[theorem]{Corollary}
\newtheorem{proposition}[theorem]{Proposition}

\theoremstyle{definition}
\newtheorem{definition}[theorem]{Definition}

\newtheorem{remark}[theorem]{Remark}
\newtheorem{example}[theorem]{Example}
\newtheorem{notation}{Notation}

\newcommand{\free}{{\rm Free}}

\def\pr{\mathbb{P}}
\def\prev{\mathbb{P}}

\newcommand{\CA}{\mathcal{C}({\bf A})}

\newcommand{\no}[1]{\widebar{#1}}
\newcommand{\at}{{\rm at}}
\newcommand{\TA}{\mathcal{T}({\bf A})}

\title{Compound conditionals as random quantities and Boolean algebras}

\author{Tommaso Flaminio$^1$, 
Angelo Gilio$^2$, 
Lluis Godo$^1$, \\
Giuseppe Sanfilippo$^3$ \vspace{1cm} \\
$^1$ Artificial Intelligence Research Institute (IIIA), CSIC \\ 08193 Bellaterra, Spain \\
{\tt \{tommaso,godo\}@iiia.csic.es} \\
$^2$ Dept. SBAI, University of Rome ``La Sapienza'', \\ Rome, Italy (retired)\\
{\tt angelo.gilio@sbai.uniroma1.it}\\
$^3$ Dept. Mathematics and Computer Science, University of Palermo, \\ Palermo, Italy \\
{\tt giuseppe.sanfilippo@unipa.it}
}
\date{} 

\begin{document}

\maketitle              % typeset the header of the contribution

\maketitle

\begin{abstract}
 Conditionals play a key role in different areas of logic and probabilistic reasoning, and they have been studied and formalised from different angles. In this paper we focus on the de Finetti's notion of conditional as a three-valued object, with betting-based semantics, and its related approach as random quantity as mainly developed by two of the authors. Compound conditionals have  been studied in the literature, but not in full generality. In this paper we provide a natural procedure to explicitly attach  conditional random quantities to arbitrary  compound conditionals that also allows us to compute their previsions. By studying the properties of these random quantities, we show that, in fact, the set of compound conditionals can be endowed with a Boolean algebraic structure. In doing so, we pave the way to build a bridge between the long standing tradition of three-valued conditionals and a more recent proposal of looking at conditionals as elements from suitable Boolean algebras. 
\end{abstract}

\section{Introduction}

Conditional expressions are pivotal in representing knowledge and reasoning abilities of intelligent agents. Conditional reasoning features in a wide range of areas spanning non-monotonic reasoning  \cite{adams75,DuboisP94,benferhat97,Isberner2001,gilio02,gilio13,BeierleEKK18}, causal inference \cite{Halpern2,Rooij}, and more generally reasoning under uncertainty \cite{Halpern2003,coletti02,PfSa17,SPOG18} or  conditional preferences \cite{Ghirardato,Vantaggi2010,ColettiPV19}.
 Bruno de Finetti was one of the first who put forward an analysis of conditionals beyond the realm of conditional probability theory arguing that they cannot be described within the bounds of classical logic 
\cite{deFi35,definetti37-2}. He expressed this by referring to them as {\em trievents}: a conditional, denoted as $(a\mid b)$, is a basic object to be read ``$a$ given $b$'' that turns out to be true if both $a$ and $b$ are true, false if $a$ is false and $b$ is true, and {\em void} if $b$ is false.

%In this setting, and more precisely, de Finetti
%allowed the evaluation of conditional events to be a \emph{partial}
%function. This is illustrated clearly by referring to the \emph{betting
%  interpretation} of subjective probability. To illustrate this, fix an
%{\em interpretation} $w \in \Omega)$. Then de Finetti interprets
%a conditional event ``$a$ given $b$'' as follows: a bet on
%\begin{equation*}
%(a \mid b)  \text { is }
%\begin{cases}
%\text{ won  } & \text{ if }  w(b)=w(a)=1;\\
%\text{ lost } & \text{ if  }  w(b)=1 \text { and } w(a)=0;\\
%\text{ called-off }  & \text{ if  }  w(b)=0.
%\end{cases}
%\end{equation*}
%Read in logical terms, under the given interpretation $w$, a conditional can thus be  `true' (gets value 1) when $w(b)=w(a)=1$, or `false' (gets value 0) when $w(b)=1$ and $w(a)=0$, or else `void' (gets no truth-value) when $w(b) = 0$. Actually,  when $w(b) = 0$, it has been argued in \cite{} in terms of a coherent  behaviour that the value the conditional takes is equated to the conditional probability $P(a \mid b)$.

The vast literature on conditionals also includes the study of {\em compound} conditionals, that is to say, those expressions obtained by combining basic conditionals like $(a\mid b)$ by usual logical operations of ``and'', ``or'', ``negation'' and so forth, see, e.g.,  \cite{Scha68,Cala87,McGe89,GoNW91,Jeff91,edgington95,Miln97,NgWa94,StJe94,Kauf09,CiDu13}.

In this line, and based on de Finetti's original conception, \cite{GiSa13a,GiSa14,GiSa19,GiSa20,GiSa21A,GiSa21} propose  and develop  an approach to interpret both basic and compound conditionals as random quantities. %  (see, e.g. \cite{GiSa13a,GiSa14,GiSa19,GiSa20}). 
%This approach allows for a {\em numerical representation} of conditionals and the compound expressions that can be generated within their language. 
%Indeed, as we will recall in Section \ref{secPreliminaries}, starting from trivents $(a\mid b)$ regarded as a random quantity defined on the set $\Omega$ of all classical evaluation and taking value in a three-valued set $\{0,1,y\}$, one can extend such representation also to cover more complex conditional formulas.
This approach has been proved to allow for a suitable {\em numerical representation} of conditionals and some families of compound conditional expressions. 
%that can be generated within their language. 
Indeed, as we will recall in Section \ref{secPreliminaries}, starting from trivents $(a\mid b)$ regarded as random quantities %defined on the set $\Omega$ of all classical evaluation and
taking values in a three-element set $\{0,1,x\}$, where $x$ is a real value in $[0, 1]$ representing a conditional probability, one can extend such representation also to cover more complex conditional formulas.

%\blue{anything about measure-free conditionals?}

An alternative, more logically-oriented approach to 
%basic and compound 
conditionals has been put forward in e.g.\ \cite{vanFraassen1976261} or \cite{Goodman}'s Conditional Event algebras,  and more recently in \cite{FlGH20} in a finitary context. These papers  formalise conditional expressions in an algebraic setting, and therefore  provide a {\em symbolic representation} of them. In this approach, as it is common in logico-algebraic representations, neither basic expressions $(a\mid b)$ nor compound conditional formulas necessarily have a numerical counterpart, as their interpretation remains at the symbolic level. However, this does not forbid, as it is also shown in \cite{FlGH20}, to consider for instance a fully compatible  probabilistic layer on top of it.  

In the present paper we will put forward an analysis that  takes inspiration from both of the above settings and proposes a numerical representation of conditionals that, in addition, also allows for a logico-symbolic representation. 
In particular, we will:
\begin{itemize}
    \item present a natural and uniform procedure to interpret compound conditionals as random quantities;
    \item investigate the numerical and logical properties of such a representation for compound conditionals via their associated random quantities; and
    \item prove that, at the logical level, the present approach, leads to the same results of \cite{FlGH20}, where the authors showed that compound symbolic conditionals can be endowed with  a Boolean algebra structure.
\end{itemize}
Let us further remark that the apparent contradiction between the  perspective that looks at three-valued conditionals as numerical random quantities as done, e.g., in \cite{GiSa14,GiSa21A}, and the Boolean algebraic perspective on conditionals used in \cite{FlGH20} to reason about them, actually dissolves once we precisely set at which level the numerical and the symbolic representation intervene. 

The present paper is organized as follows. In the next Section \ref{secPreliminaries} we will recall in some more details the original approach of \cite{deFi35} (Subsection \ref{SEC:DEFIN}) and the one that followed of \cite{GiSa14} (Subsection \ref{sec22}). 
%Furthermore, in Subsection \ref{sec23} we will clarify the notation used. 
The interpretation of compound conditionals in terms of conditional random quantities will be the subject of Section \ref{secDefinitions}. There we will also provide examples in order for our basic construction to be clear. In Section \ref{secProperties} we will then prove numerical and also logical properties of the random quantities that represent compound conditionals, and in Section \ref{secProb} we will examine some probabilistic aspects. The comparison between the algebras arising from the random quantities studied here and the Boolean algebras of conditionals of \cite{FlGH20} will be the topic of Section \ref{secComparison}, where we  prove that they
%latter structures
are 
indeed isomorphic. In the last Section \ref{SecConclusion} we will gather some conclusions and remarks about future work on this subject.

\section{Some preliminary comments on conjoined conditionals}\label{secPreliminaries}
In this section, in order to better understand the formalism and results of this paper, we recall some notions given in the approach by Gilio-Sanfilippo, and their relation with the notions and results given here.

We will denote here events by capital letters, such as $A,B,H,K,\ldots $.
%The indicator of an event $A$, denoted by the same symbol, is 1, or 0, according to whether $A$ is true, or false.
Moreover, we denote random quantities by capital letters, such as $X,Y,Z\ldots$. In particular we will denote by $X_{A}$ the indicator function of the event $A$. 
 
 \subsection{The conditional prevision in the approach of de Finetti} \label{SEC:DEFIN}
%CHECK THE NEGATION SYMBOL, THE SYMBOLS OF $\at({\bf A}), \top, \emptyset$..\\  

Given a random quantity $Z$ and an event $H \neq \bot$, in the approach of de Finetti (1935) the conditional prevision $\pr(Z|H)$ can be assessed by applying the following conditional bet: 
\begin{itemize}
    \item[1)]  you are asked to assess the value $z=\pr(Z|H)$, by knowing that if $H$ is true then the bet is in effect;
    \item[2)] if the bet has effect, then you pay $z$ and you receive the random amount $Z$; otherwise, if $H$ is false, the bet is null; 
    \item[3)] the checking of coherence for the assessment $z=\pr(Z|H)$ is made by only considering the cases where the bet has effect.
\end{itemize}
We observe that for the random quantity $Z{\cdot} X_H + y X_{\no{H}}$,  by linearity of the prevision,  it holds that, for all $y$,
\[
\begin{array}{l}
\pr(Z{\cdot} X_H + y X_{\no{H}} \mid H)=\pr(Z{\cdot}X_H \mid H)+yP(\no{H}\mid H)\\=\pr(Z\mid H)=z.
\end{array} 
\]
In particular, for $y=z$, one has 	$\pr(Z X_H + z X_{\no{H}} \mid H)=\prev(Z \mid H)=z$.
		 Moreover, for all $y$, 
$$
\begin{array}{l}
\pr(Z{\cdot}X_H + y X_{\no{H}}) =
 \pr(Z{\cdot}X_H) + y P(\no{H}) = \\ =\pr(Z \mid H)P(H) + y P(\no{H}) = zP(H) + y P(\no{H}) \,, 
\end{array}
$$
and hence,  for $z = y$, we have
$$
\begin{array}{l}
\pr(Z{\cdot}X_H + z X_{\no{H}}) = zP(H) + z P(\no{H}) = z. %\\ 
%=\pr(Z|H)P(H) + z P(\no{H}) = z \,.
\end{array}
$$
In other words, we have the following equalities: 
\begin{equation}\label{EQ:PREVITER}
\pr(Z{\cdot}X_H + z X_{\no{H}})=\pr(Z{\cdot}X_H + z X_{\no{H}} \mid H)=\pr(Z\mid H)=z.	
\end{equation}	
Now, consider  the random quantity $Z{\cdot}X_H + y X_{\no{H}}$, under the assumptions that: 
\begin{itemize}
\item[(i)] $y=\pr(Z{\cdot}X_H + y X_{\no{H}})$;
\item[(ii)]  in order to check the coherence of the assessment $y$, we discard the cases where you receive back the paid amount $y$ (whatever $y$ be), that is we discard the case where $H$ is false (bet called off).
\end{itemize}
Of course,  $y=z$  satisfies both $(i)$ and $(ii)$.
\begin{itemize}
\item[] {\em Question: is it coherent to assess $y \neq z$?} 
\end{itemize}
We  show below that the  answer is NO. \\ 
Indeed, if we make a bet on the random quantity 
\[
(Z{\cdot}X_H + z X_{\no{H}}) - (Z{\cdot}X_H + y X_{\no{H}}) = (z-y)X_{\no{H}},
\]
we should agree to pay $z-y$ by receiving $(z-y)X_{\no{H}} \in \{0, z-y\}$. As when   $H$ is false we receive back the paid amount $z-y$ (whatever $z-y$ be), this case must be discarded for checking the coherence of the assessment $z-y$. 
Then, we should pay $z-y$ by knowing that (when the bet is not called off) we receive 0. Therefore, by coherence,   $z-y=0$, that is $y=z=\pr(Z|H)$. 

In other words, in order to assess the conditional prevision $\pr(Z|H)$, it is equivalent to assess the value $z$ such that $z=\pr(Z{\cdot} X_H+zX_{\no{H}})$; then,
 we can define the conditional random quantity $Z|H$  as 
\begin{equation}\label{EQ:ZgH}
Z|H=ZX_H+zX_{\no{H}}, \text{ where } z=\prev(ZX_H+zX_{\no{H}}).
\end{equation}
In this way we can look at $z$ both as the conditional prevision $\prev(Z|H)$ \emph{\`a la de Finetti},  and as the prevision of  the conditional random quantity  $Z|H=ZX_H+zX_{\no{H}}$. In particular, given two events $A$ and $H \neq \emptyset$, if $x=\pr(X_{AH} + x X_{\no{H}})$, then $x=P(A|H)$ and hence
\begin{equation}\label{EQ:INDICATORAgH}
X_{A|H} =X_{AH} + P(A|H) X_{\no{H}}. 
\end{equation}

A final remark:  from (\ref{EQ:PREVITER}) it holds that
\begin{equation}\label{EQ:EQ:PREVITERBIS}
\prev[(Z|H)|H]=\pr(ZX_H + z X_{\no{H}} \mid H)=\prev(Z|H)=z;
\end{equation}
then we formally obtain 
$$
\begin{array}{l}
(Z|H)|H= \\ = (ZX_H+zX_{\no{H}}) X_H+\pr(Z X_H + z X_{\no{H}}\mid H) X_{\no{H}}=\\ =ZX_H+z X_{\no{H}}=Z|H.
\end{array}
$$

\subsection{The conjunction in the approach by Gilio $\&$ Sanfilippo}\label{sec22}
We recall that in the approach by Gilio and Sanfilippo the compound conditionals, like conjunctions and disjunctions, are defined as conditional random quantities, in the setting of coherence, see e.g. \cite{GiSa14,GiSa19,GiSa21}. In this section, in order to make explicit the numerical aspects, we recall these notions by using the notations of the current paper. Then, the indicator of an event $A$, or a conditional event $A|H$ is denoted (not by the same symbol, but) by $X_A$, or $X_{(A|H)}$, respectively. Likewise, the conjunction $A|H\wedge B|K$ is denoted by $X_{(A|H)\wedge(B|K)}$, and so on.

 In the setting of coherence, given a probability assessment $P(A|H)=x, P(B|K)=y$, the conjunction of  $A|H$ and $B|K$ is defined as 
\begin{equation}\label{EQ:CONG}
X_{(A|H) \wedge (B|K)} = \left\{
\begin{array}{ll}
1, & \mbox{if $AHBK$ is true,} \\
0, & \mbox{if $\no{A}H {\vee} \no{B}K$ is true,} \\
x, & \mbox{if $\no{H}BK$ is true,} \\
y, & \mbox{if $AH\no{K}$ is true,} \\
z, & \mbox{if $\no{H}\no{K}$ is true,} \\
\end{array}
\right.
\end{equation}
where $z$ is the prevision of $X_{(A|H) \wedge (B|K)}$, which (in the framework of the betting scheme) is the amount to be paid in order to receive the random amount $X_{(A|H) \wedge (B|K)}$. 
We observe that 
\begin{equation}\label{eqfirstConj}
X_{(A|H) \wedge (B|K)} = X_{AHBK} + x X_{\no{H}BK} + y X_{AH\no{K}} + z X_{\no{H}\no{K}
} \,,\;
\end{equation}
where $z=\prev( X_{AHBK} + x X_{\no{H}BK} + y X_{AH\no{K}} + z X_{\no{H}\no{K}})$.
Then, by applying (\ref{EQ:ZgH}), where  $H$ is replaced by $H{\vee} K$ and  $Z$ is replaced by $X_{AHBK} + x X_{\no{H}BK} + y X_{AH\no{K}}$, we have 
$$
\begin{array}{l}
Z\cdot X_{H\vee K}= (X_{AHBK} + x X_{\no{H}BK} + y X_{AH\no{K}}) X_{H \vee K}=\\
= X_{AHBK} + x X_{\no{H}BK} + y X_{AH\no{K}},
\end{array}
$$
and it follows that 
\[
\begin{array}{ll}
X_{AHBK} + x X_{\no{H}BK} + y X_{AH\no{K}} + z X_{\no{H}\no{K}}
= \\ =(X_{AHBK} + x X_{\no{H}BK} + y X_{AH\no{K}})|(H{\vee} K).
\end{array}
\]
Hence, from (\ref{eqfirstConj}) we obtain
\[
X_{(A|H) \wedge (B|K)} =(X_{AHBK} + x X_{\no{H}BK} + y X_{AH\no{K}})|(H\vee K),
\]
with 
\begin{equation}\label{EQ:CONDPREV}
	\begin{array}{l}
z=\prev[(X_{AHBK} + x X_{\no{H}BK} + y X_{AH\no{K}}) \mid H{\vee} K)=\\
\quad =P(AHBK \mid H {\vee} K)+xP(\no{H}BK\mid H{\vee} K) \; + \\ \qquad + \; yP(AH\no{K}\mid H{\vee} K).
\end{array}
\end{equation}
We observe that, if $P(H\vee K)>0$, 
\[
z=\frac{P(AHBK) + x P(\no{H}BK) + y P(AH\no{K})}{P(H \vee K)} ,
%\,,\; \text{ if } P(H\vee K)>0,
\]
which is the well known formula given in   \cite{Kauf09} and in \cite{McGe89}. 

\noindent Moreover,	 by observing that 
\[
\begin{array}{l}
X_{(AHBK\mid H\vee K)}= X_{AHBK}+P(AHBK \mid H {\vee} K) X_{\no{H}\no{K}},\\
xX_{(\no{H}BK\mid H\vee K)}=x X_{\no{H}BK}+xP(\no{H}BK\mid H{\vee} K) X_{\no{H}\no{K}},\\
yX_{(AH\no{K}\mid H\vee K)}=yX_{AH\no{K}}+yP(AH\no{K} \mid H{\vee} K) X_{\no{H}\no{K}},\\
\end{array}
\]
from (\ref{EQ:CONDPREV}) we obtain 
\[
\begin{array}{l}
	X_{(AHBK \mid H\vee K)} +
	xX_{(\no{H}BK \mid H\vee K)}+
	yX_{(AH\no{K} \mid H\vee K)} \\
	= X_{AHBK} + x X_{\no{H}BK} + y X_{AH\no{K}} + z X_{\no{H}\no{K}}.
\end{array}
\]
Finally,
\[
\begin{array}{l}
X_{(A|H) \wedge (B|K)}
= \\ X_{(AHBK \mid H \vee K)} + x X_{(\no{H}BK\mid H \vee K)} + y X_{(AH\no{K}\mid H \vee K)}.
\end{array}
\]

 \subsection{The relation with the approach in this paper}\label{sec23}
For notational convenience, in the rest of the paper we will use lower case letters for  events involved in (compound) conditionals; hence, for instance, we will denote by $(a|b)$ a conditional event and by $(a|b) \land (c|d)$ the conjunction of $(a|b)$ and $(c|d)$. Moreover, in order to distinguish the logical aspects from the numerical ones, with each compound conditional $t$ we will associate a suitable random quantity $X_t$, whose numerical values are the conditional previsions of some intermediate objects, called reducts. In particular, with the conjunction  $(a|b) \land (c|d)$ we will associate a random quantity $X_{(a|b) \land (c|d)}$, whose  numerical values are conditional previsions  associated with the reducts determined by the  following  partition which corresponds to the one in \eqref{EQ:CONG}:
 \[
\{ abcd \,,\; \no{a}b \lor \no{c}d \,,\; \no{b}cd \,,\; ab \no{d} \,,\; \no{c}\no{d} \}\,.
 \]
 As we will see more in general later, the reducts associated with the elements of the partition above are denoted by
 \[
 1 \,,\; 0 \,,\; (a|b) \,,\; (c|d) \,,\; (a|b) \wedge (c|d) \,,
 \]
respectively. From them we obtain the respective numerical values $1,0,x,y,z$ of $X_{(a|b) \land (c|d)}$,  which are interpreted by the following  conditional previsions:
$$ 
\begin{array}{c}
 1 = \pr(1|\Omega) \,,\; 0 = \pr(0|\Omega) \,,\;   x = \pr(X_{(a|b)}|b) \,,\;\\ y =\pr(X_{(c|d)}|d),\;
 z = \pr[X_{(a|b) \wedge (c|d)}|(c \lor d)] \,.
 \end{array}
$$
We observe that, based on (\ref{EQ:ZgH}), it could be verified that
$$ 
\begin{array}{l}
x=\pr(X_{(a|b)}|b)=P(a|b) \,,\; y =\pr(X_{(c|d)}|d)= P(c|d),\\
z=\prev[X_{(a|b) \wedge (c|d)}|(c \lor d)] =
\prev[X_{(a|b) \wedge (c|d)}].
\end{array}
$$
Thus,  $X_{(a|b) \wedge (c|d)}$  
is nothing else that  the conjunction between the  conditional events $(a|b)$ and $(c|d)$ defined  in (\ref{EQ:CONG}).
 We remark that, by this approach, each compound conditional $t$ %of the conditional algebra 
 is not explicitly defined, but we operate on it by means of the associated random quantity $X_t$.

%\newpage

\section{General compound conditionals as conditional random quantities}\label{secDefinitions}

From now on we will be considering a fix {\em finite} Boolean algebra of ordinary events ${\bf A} = (A, \land, \lor, \neg, \bot, \top)$. For the sake of a lighter notation, we will also use $ab$ for the conjunction  $a\land b$ of events $a$ and $b$, $\bar a$ for $\neg a$ of the event $a$, while we will keep denoting the disjunction of $a$ and $b$  by $a\vee b$. 

In this setting, the set of the atoms $\at({\bf A})$ of the algebra of events $\bf A$, can be identified with the set $\Omega$ of interpretations for $\bf A$, i.e.\ the set of Boolean algebra homomorphisms $w: {\bf A} \to \{0, 1\}$. Indeed, we will say that an event $a \in \bf A$ is {\em true} (resp. {\em false}) under an interpretation (or possible world) $w \in \Omega$ when $w(a) = 1$ (resp. $w(a) = 0$), also denoted as $w \models a$ (resp. $w \not \models a$).

As we anticipated in the previous section, conditional events of the form ``if $a$ then $b$', or ``$a$ given $b$'', where $a$ and $b$ are events from $\bf A$ with $b$ different from $\bot$, will be denoted by pairs $(a \mid b)$. We will also let $A|A' = \{ (a \mid b) \mid a \in A, b \in A'\}$, where $A' = A \setminus \{\bot\}$, be the set of all conditionals that can be built from $\bf A$, that will be also called  {\em basic conditionals}. 
 {\em Compound conditionals}, then,  are combinations of basic ones by operations of negation, conjunction and disjunction, that we will keep denoting them as $\neg, \land$ and $\lor$ without danger of confusion. %with the operations on ordinary events. In fact, we d
 Denote by $\mathbb{T}(A)$  the term algebra of type $(\wedge, \vee, \neg, \bot,\top)$ over the set $A|A'$, so that its support  $\mathbb{T}(A)$ is
 %= (\mathbb{T}(A), \land, \lor, \neg, \top, \bot)$ be 
 the  set of arbitrary terms generated from $A|A'$ (taken as variables) over the signature $(\neg, \land, \lor, \bot,\top)$. For instance, if $a, c, e \in A$ and $b, d, f \in A'$, then $(a \mid b) \land (c \mid d)$ or $(a \mid b) \lor ((c \mid d) \land \neg(e \mid f))$ are examples of compound conditionals from $\mathbb{T}(A)$. 
  One of our  ultimate aims will be to present a Boolean algebra obtained from ${ \mathbb{T}(A)}$ that extends $\bf A$.

%As per the pay-offs of the bet, it has been shown that the bet is coherent (i.e. whenever there is no strategy for the bookmaker ensuring a sure loos or a sure gain) when 

In the rest of the section we will further extend the random quantity-based approach to conditionals and propose an unambiguous procedure to interpret any compound conditional as a suitable random quantity.
Notice that, given a family of $n$ conditional events $\{(a_1\mid b_1), \ldots, (a_n\mid b_n)\}$, we can consider the possible values of the random vector $(X_{(a_1|b_1)},\ldots, X_{(a_n|b_n)})$, the so-called points $Q_h$’s, by means of which we can develop a geometrical approach for checking coherence and for propagation (see, e.g., \cite{GiSa20,GiSa21}).

%\subsection{A general setting for compound contidionals}

For every $t \in \mathbb{T}(A)$, let us denote by $Cond(t) = \{(a_1 \mid b_1), \ldots, (a_n \mid b_n)\}$ the set of basic conditionals appearing in $t$, and by ${\bf b}(t) = b_1 \lor \ldots \lor b_n$ the disjunction of the antecedents in $Cond(t)$. 

\begin{definition}
Let $w \in \Omega$ be a classical interpretation and let  $t \in \mathbb{T}(A)$ be a term. The {\em $w$-reduct} of $t$, denoted $t^w$, is the term in $\mathbb{T}(A)$, called the {\em $w$-reduct of $t$}, obtained from $t$ by the following procedure: 
\begin{itemize}
\item[(1)] replace each $(a_i \mid b_i) \in Cond(t)$ by 1 if $w \models a_i  b_i$, and by 0 if $w \models \bar a_i  b_i$, 
\item[(2)] apply the following reduction rules to subterms of $t$ until no further reduction is possible: 

- $\neg 1 := 0$, $\neg 0 := 1$ \\
- $r {\land} 1 = 1 {\land} r := r$ \\
- $r {\land} 0 = 0 {\land} r := 0$ \\
- $r {\lor} 1 = 1 {\lor} r := 1$ \\
- $r {\lor} 0 = 0 {\lor} r := r$ 

where the $r$ denotes a subterm of $t$. 
\end{itemize}
\end{definition}

%\noindent Note that if $w \models \neg {\bf b}(t) $ then $t^w = t$. 

This symbolic reduction procedure has some interesting properties. First of all, $w \models \neg {\bf b}(t)$, then no reduction is possible and hence $t^w = t$. Second, the reduction commutes with the operation symbols, in the following sense: 
\begin{lemma} \label{reduct} For every terms $t\in \mathbb{T}(A)$ and for every $w\in \Omega$, the following hold:
\begin{enumerate}
\item $(\neg t)^w = \neg t^w$
\item $(t {\land} s)^w = t^w {\land} s^w$
\item $(t {\lor} s)^w = t^w {\lor} s^w$
\end{enumerate}
\end{lemma}

We will denote by $Red(t) = \{t^w \mid w \in \Omega$ the set of $w$-reducts of $t$, and by $Red^0(t) = Red(t) \setminus \{t\}$, the set of its {\em proper $w$-reducts}.

\begin{example} Let $t = (a \mid b) {\land} ((c \mid d) {\lor} \neg(e \mid f))$ and let $w$ such that $w(a) = 1, w(b) = 0, w(c) = 0, w(d) = 0, w(e) = 1 , w(f) = 1$, i.e. $w \models a   \bar b  \bar c  \bar d  e f$. Then 
\begin{center}
$t^w := (a \mid b) {\land} ( (c \mid d) {\lor} \neg 1) :=  (a \mid b) {\land} ( (c \mid d) {\lor} 0)$ \\ $:=  (a \mid b) {\land}   (c \mid d).$
\end{center}
Let $w'$ such that $w' \models a  b  c  \bar d  e  f$. Then 
$$t^{w'} := 1 {\land} ((c \mid d) {\lor} \neg 1) := (c \mid d) {\lor} 0 := (c \mid d). $$
Further, if $w''$ is such that $w'' \models abcdef$, then $$t^{w''} := 1 \land (1 \lor \neg 1) := 1 \land (1 \lor 0) := 1 \land 1 := 1.$$
In fact, one can check that $Red^0(t) = \{1, 0, (a \mid b),$ \\ $ (c \mid d), \neg(e \mid f), (a \mid b) {\land} (c \mid d), (a \mid b) {\land} \neg(e \mid f),$ \\ $  (c \mid d) {\lor} \neg(e \mid f) \}.$
\end{example}

We recall that, given a finite algebra ${\bf A}$ and a conditional probability  $P:A\times B\rightarrow [0,1]$, where $B\subseteq A'$,  $P$ could  not be coherent \cite{GiSp92,Gili95a,coletti02}.
Some sufficient conditions for coherence of $P$
are: $(i)$ $B=A'$; $(ii)$ $B\cup \{\bot\}$ is a subalgebra of $A$; $(iii)$ $B$ is an additive
\cite{Holz85,coletti02} or quasi-additive
\cite{gili89,Rigo88,Sanf12}
 class. 
 
Now, we recall the notion of conditional prevision of a random quantity. 
\begin{definition}
Let $X: \Omega \to [0, 1]$ be a random quantity, and let $b\in A$ be an event. Then, given a  conditional probability $P: A \times A' \to [0, 1]$,  the {\em conditional
prevision of $X$ given $b$} is defined as:
$$\mathbb{P}(X \mid b) = \sum_{w \in \Omega} X(w) \cdot P(w \mid b) =  \sum_{w\in b} X(w) \cdot P(w \mid b).$$
\end{definition}

The next definition  presents a suitable way to associate a random quantity to every compound conditional $t  \in \mathbb{T}(A)$. %\\

%\noindent
%{\bf Notation}: Let $t \in \mathbb{T}(A)$ and let $Cond(t) = \{(a_1 \mid b_1), \ldots, (a_n \mid b_n)\}$. Then we will denote by ${\bf b}(t)$ the event $b_1 \lor \ldots \lor b_n$. 

\begin{definition} \label{xt} Let ${\bf A}$ be a finite Boolean algebra and  $P: A \times A' \to [0, 1]$ a conditional probability on ${\bf A}$.
%a coherent prevision assignment $\mathbb{P}: Red(t) \to [0, 1]$, 
For every term $t$ in $\mathbb{T}(A)$, we define the random quantity $X_t: \Omega \to [0, 1]$ as follows: for every $w \in \Omega$, 
$$X_t(w) := \mathbb{P}(X_{t^w} \mid {\bf b}(t^w)).$$
If $t^w = 1$ or $t^w = 0$, we take ${\bf b}(t^w) = \top$, and hence $X_1 = 1$ or $X_0 = 0$  respectively. Thus,  $X_1(w) = 1$ or $X_0(w) = 0$. 
\end{definition}
Regarding this definition, some observations are in order here: 
%\begin{remark}\label{remFirst}
\begin{itemize}
    \item[(i)] As we have observed before, if $w \models \no{\bf b}(t)$ then $t^w = t$, and thus $X_t(w) = \mathbb{P}(X_{t} \mid {\bf b}(t))$, and hence $ \mathbb{P}(X_{t} \mid \no{\bf b}(t))=\mathbb{P}(X_{t} \mid {\bf b}(t))$.
%    does not really depend on $w$. 

\item[(ii)] The above definition of $X_t$ strongly depends on the assumed conditional probability $P$ of $A\times A'$. Actually, once we fix the initial conditional probability $P$ on $A\times A'$, all random quantities $X_t$ are fully determined.  %Fix a probability, all $X_t$ are determined. 
Indeed, the above Definition \ref{xt}  is in fact a recursive definition, since $$\mathbb{P}(X_{t^w} \mid {\bf b}(t^w)) = \sum_{w'} X_{t^w}(w') \cdot P(w' \mid {\bf b}(t^w)),$$
and in turn, 
$$X_{t^w}(w') =\mathbb{P}(X_{(t^w)^{w'}} \mid {\bf b}((t^w)^{w'})),$$
and so on, until reaching basic conditionals.

\item[(iii)] As a consequence, 
%Notice that, for every $t\in \mathbb{T}(A)$, the above definition of $X_t$ strongly depends on the assumed conditional probability $P$ of ${\bf A}$. In other words, 
in general different conditional probabilities $P$ and $P'$ on ${\bf A}$ will define different random quantities $X_t$ for the same term $t$.
%\vspace{.1cm}

\item[(iv)] In the case $t$ is of the form $t = (a \mid \top)$, the associated random quantity $X_t$ coincides with the indicator function of $a$, that is, $X_t(w) = 1$ whenever $w \models a$, and $X_t(w) = 0$ otherwise. This shows that $(a \mid \top)$ can indeed be indentified with the plain event $a$. In this case, for the sake of a lighter notation, we will write $X_a$ instead of $X_{(a \mid \top)}$. 
%\end{remark}
\end{itemize}

\begin{notation} From now on, for simplicity, for any $t \in  \mathbb{T}(A)$, we will write $\mathbb{P}^c(X_t)$ for $\mathbb{P}(X_t \mid {\bf b}(t))$. 
\end{notation}

Displayed in another way, the random quantity $X_t$ can be specified as follows: let 
$Red^0(t) = \{ t^w \mid w \in \Omega\} = \{ t_1, t_2, ..., t_k\}$ 
and let $\{E_1, E_2, ..., E_k\}$ be the corresponding subsets of interpretations leading to a same element of $Red^0(t)$, then $$X_{t}(w) = \mathbb{P}(X_{t^w} \mid {\bf b}(t^w)) =$$
$$%X_{t}(w) = \mathbb{P}(X_{t^w} \mid {\bf b}(t^w)) 
= \left \{
\begin{array}{ll}
\mathbb{P}^c(X_{t_1}), & \mbox{if } w \models E_1 \\
\ldots, &  \ldots \\
\mathbb{P}^c(X_{t_k}), & \mbox{if } w \models E_k  \\[0.1cm] \hdashline  \\[-0.3cm]
\mathbb{P}^c(X_t), & \mbox{if } w \models \neg (E_1 \lor \ldots \lor E_k)
\end{array}
\right . 
$$
where the dashed line separates those cases where the interpretation $w$ belongs to ${\bf b}(t)$ from those which do not.  

In this setting, it is clear that $X_t$ is in fact the following linear combination of the indicator functions of the events defining an associated partition: 
$$X_t = \mathbb{P}^c(X_{t_1}) X_{E_1} + \ldots + \mathbb{P}^c(X_{t_k}) X_{E_k} + \mathbb{P}^c(X_t) X_{E_{k+1}}$$
where $E_{k+1} = \no{E_1} \land \ldots \land \no{E_k} = \overline{{\bf b}}(t)$, and hence
$$
\mathbb{P}^c(X_t) = \mathbb{P}^c(X_{t_1}) \cdot P(E_1 \mid  {\bf b}(t)) + \ldots +  \mathbb{P}^c(X_{t_k}) \cdot P(E_k \mid  {\bf b}(t)) .
$$
We remark that, from (i) above and from (\ref{EQ:PREVITER}), %as $\prev(X_t\mid {\bf b}(t))=\prev(X_t\mid \no{{\bf b}}(t))=\prev^c(X_t)$, 
it holds that
\begin{equation}\label{EQ:PREVXT}
\begin{array}{ll}
\prev(X_t)=
\prev(X_t\cdot X_{{\bf b}(t)})+\prev(X_t\cdot X_{\no{\bf b}(t)})=\\
=\prev(X_t\mid {\bf b}(t))
P({\bf b}(t))
+
\prev(X_t\mid \no{{\bf b}}(t))
P(\no{{\bf b}}(t))=
\\
=\prev^c(X_t)
P({\bf b}(t))
+
\prev^c(X_t)
P(\no{{\bf b}}(t))=\prev^c(X_t).
\end{array}
\end{equation}
In other words, formula (\ref{EQ:PREVXT}) shows that
the prevision of $X_t$ coincides with the conditional prevision $\prev^c(X_t)$.

We end this section by exemplifying  the above definition of $X_t$ for  selected known cases of compound conditionals $t$ that will be helpful in next sections.
 
\begin{example} \label{firstex} Let $t = (a \mid b)$. Then, by applying the above definition, we get
$$t^w = \left \{
\begin{array}{ll}
1, & \mbox{if } w \models a  b \\
0, & \mbox{if } w \models \bar a  b \\
(a \mid b), & \mbox{if } w \models \bar b \\
\end{array} 
\right .  
\quad
 {\bf b}(t^w) = \left \{
\begin{array}{ll}
\top, & \mbox{if } w \models a  b \\
\top, & \mbox{if } w \models \bar a  b \\
b, & \mbox{if } w \models \bar b \\
\end{array} 
\right .  
$$
and thus we have: 
$$X_{(a \mid b)}(w) = \mathbb{P}(X_{t^w} \mid {\bf b}(t^w)) =$$ $$= \left \{
\begin{array}{ll}
\mathbb{P}(X_{1} \mid \top) = 1, & \mbox{if } w \models a  b \\
\mathbb{P}(X_{0} \mid \top) = 0, & \mbox{if } w \models \bar a  b \\
\mathbb{P}(X_{(a \mid b)} \mid b), & \mbox{if } w \models \bar b \\
\end{array}
\right . .
$$
Now, since $P(w \mid b) = 0$ whenever $w \models \bar b$, we have 
$$
\begin{array}{l}
\mathbb{P}(X_{(a \mid b)} \mid b) =  \\ 
= 1\cdot P(a  b \mid b) + 0 \cdot P(\bar a  b \mid b) + \mathbb{P}(X_{(a \mid b)} \mid b) \cdot 0 = \\ 
= P(a  b \mid b) = P(a \mid b).
\end{array}
$$
 Therefore we get the following well-known three-valued representation of a conditional $(a \mid b)$:
$$X_{(a \mid b)}(w) =\left \{
\begin{array}{ll}
1, & \mbox{if } w \models a  b \\
0, & \mbox{if } w \models \bar a  b \\
P(a \mid b), & \mbox{if } w \models \bar b \\
\end{array}
\right . .
$$
Equivalently,  in agreement with  (\ref{EQ:INDICATORAgH}), $X_{(a \mid b)}$ can be expressed as the following linear combination of the indicator functions of the events $ab$ and $\no b$: 
$$X_{(a \mid b)}=1X_{ab}+0X_{\no{a}b}+P(a|b)X_{\no{b}} = X_{ab}+P(a|b)X_{\no{b}}.$$
\end{example}

\begin{example}\label{exNeg} Now let $t = \neg (a \mid b)$, the negation of $(a \mid b)$. Then, by applying the above definition, we get:
$$t^w = \left \{
\begin{array}{ll}
\neg 1 := 0, & \mbox{if } w \models a  b \\
\neg 0 := 1, & \mbox{if } w \models \bar a  b \\
\neg (a \mid b), & \mbox{if } w \models \bar b \\
\end{array} 
\right .  ,
$$
$${\bf b}(t^w) = \left \{
\begin{array}{ll}
\top, & \mbox{if } w \models a  b \\
\top, & \mbox{if } w \models \bar a  b \\
b, & \mbox{if } w \models \bar b \\
\end{array} 
\right .  ,
$$
and thus we have: $X_{\neg(a \mid b)}(w) =$
$$%X_{\neg(a \mid b)}(w) = 
= \mathbb{P}(X_{t^w} \mid {\bf b}(t^w)) = \left \{
\begin{array}{ll}
0, & \mbox{if } w \models a  b \\
1, & \mbox{if } w \models \bar a  b \\
\mathbb{P}(X_{\neg(a \mid b)} \mid b), & \mbox{if } w \models \bar b \\
\end{array}
\right . .
$$
Now, since $P(w \mid b) = 0$ whenever $w \models \neg b$, we have 
$$\begin{array}{lll}
\mathbb{P}(X_{\neg (a \mid b)} \mid b) = \\
=  0\cdot P(a  b \mid b) + 1 \cdot P(\bar a  b \mid b) + \mathbb{P}(X_{\neg(a \mid b)} \mid b) \cdot 0 \\
= P(\bar a  b \mid b) = P(\bar a \mid b) = 1 - P(a \mid b).
\end{array} $$
 Therefore the final expression for $X_{\neg (a \mid b)}$ is
$$X_{\neg(a \mid b)}(w) =\left \{
\begin{array}{ll}
0, & \mbox{if } w \models a  b \\
1, & \mbox{if } w \models \bar a  b \\
1-P(a \mid b), & \mbox{if } w \models \bar b \\
\end{array}
\right . .
$$
That is to say, $X_{\neg(a \mid b)} = 1- X_{(a \mid b)} = X_{(\no a \mid b)}$. 
\end{example}

\begin{example}\label{exConjunction} Let us examine again the case of a conjunction of two conditionals $t = (a \mid b) {\land} (c \mid d)$ from the current perspective.  Here we have ${\bf b}(t) = b \lor d$, and  
$$X_t(w) = \left \{
\begin{array}{ll}
1, & \mbox{if } w \models a  b  c  d \\
0, & \mbox{if } w \models (\bar a  b) \lor (\bar c  d)  \\
\mathbb{P}^c(a \mid b), & \mbox{if } w \models \bar b  c  d\\
\mathbb{P}^c(c \mid d), & \mbox{if } w \models a  b  \bar d \vspace{0.1cm}   \\ \hdashline  \\[-0.2cm] 
\mathbb{P}^c((a \mid b) {\land} (c \mid d)), & \mbox{if } w \models \bar b  \bar d \\

\end{array}
\right .
$$
Now, we know that $\mathbb{P}^c(a \mid b)= P(a \mid b)$ and  $\mathbb{P}^c(c \mid d)= P(c \mid d)$. Then, using the above definition we get:
$$\begin{array}{ll}
 \mathbb{P}^c((a \mid b) {\land} (c \mid d))  = \mathbb{P}(X_{(a \mid b) {\land} (c \mid d)}\mid b \lor d)  \\ 
\mbox{ } = P( a  b  c  d \mid b \lor d) + P(a \mid b) \cdot  P(\bar b  c  d \mid b \lor d) + \\
\mbox{ } \quad + P(c \mid d) \cdot P(a  b  \bar d \mid b \lor d),
\end{array}
$$
that coincides, when $P(b \lor d) > 0$, with the formula given in \cite{McGe89,Kauf09}. 

Let us now consider two particular cases: 

\begin{itemize}

\item First, consider the case $a \leq b = c \leq d$. Then $t = (a \mid b) {\land} (b \mid d)$, ${\bf b}(t) = b \lor d = d$, and moreover: 

- $a  b  c  d = a b d = ad = a$

- $(\bar a  b) \lor (\bar c  d) = (\bar a  b) \lor (\bar b  d) = 
(\bar a  b d) \lor (\bar a\bar b  d) = \bar a  d$
%(\bar a \lor \bar b)  (\bar a \lor d)  (b \lor \bar b)  (b \lor d) = $ \\ 
%\hspace*{1.7cm} $ = \bar a  ( \bar a \lor d)  (b \lor d)  = \bar a  d$

- $ \bar b  c  d = \bar b  b = \bot$

- $a  b  \bar d = a  \bar d = \bot$

\noindent Hence, the above general definition reduces to: 
$$X_{(a \mid b) {\land} (b \mid d)}(w) = \left \{
\begin{array}{ll}
1, & \mbox{if } w \models a \; \\ %(= a  d)\\
0, & \mbox{if } w \models \bar a  d  \\
\mathbb{P}^c((a \mid b) \land (b \mid d)), & \mbox{if } w \models \bar d 
\end{array}
\right .
$$
where $\mathbb{P}^c((a \mid b) \land (b \mid d)) = 1 \cdot P(a \mid d) =  P(a \mid d) = \mathbb{P}^c((a \mid d))$ and thus $X_{(a \mid b) {\land} (b \mid d)}(w)  = X_{(a \mid d)}(w)$ for all $w$.  
Thus, from the numerical point of view, the compound conditional $(a \mid b) {\land} (b \mid d)$  behaves as the basic conditional $(a \mid d).$
%Thus, in a sense, the compound conditional $(a \mid b) {\land} (b \mid d)$  behaves as the basic conditional $(a \mid d).$ 
See also \cite[formulas (55) and (56)]{GiSa20} where it is also observed  that $P(a|d)=P(a|b)P(b|d)$ ({\em compound probability theorem}).
\item Consider now the case $b = d$. Then $t = (a \mid b) \land (c \mid b)$ and ${\bf b}(t) = b$. In this case the above general expression reduces to the following one: 
$$X_t(w) = \left \{
\begin{array}{ll}
1, & \mbox{if } w \models a  b  c  \\
0, & \mbox{if } w \models (\bar a \lor \bar c)  b  = \overline{a  c}  b\\
\mathbb{P}^c((a \mid b) {\land} (c \mid b)), & \mbox{if } w \models \bar b , \\
\end{array}
\right .
$$
where $\mathbb{P}^c((a \mid b) {\land} (c \mid b)) = 1 \cdot P(abc \mid b) = P(ac \mid b) = \mathbb{P}^c((ac \mid b))$, and therefore
 it is clear that $X_{(a \mid b) \land (c \mid b)}(w) = X_{(a \land c \mid b)}(w)$ for all $w$. Thus, in this case, we see that the compound conditional $(a \mid b) \land (c \mid b)$ actually behaves like the basic conditional $(a \land c \mid b)$.

\end{itemize}
\end{example}

\begin{example}\label{ex1} Finally, let us consider a more complex compound conditional $t = (a \mid b) {\land} ((c \mid d) {\lor} \neg(e \mid f))$. Again, by the above definition, we get its associated random quantity: 
%$$X_t(w) = \left \{
%\begin{array}{ll}
%1, & \mbox{if } w \models a \land b \land ((c \land d) \lor (\neg e \land f)) \\
%0, & \mbox{if } w \models (\neg a \land b) \lor (\neg c \land d \land e \land f) \\
%\mathbb{P}^c((a \mid b)), & \mbox{if } w \models \neg b \land ((c \land d) \lor (\neg e \land f))\\
%\mathbb{P}^c((a \mid b) {\land} (c \mid d)), & \mbox{if } w \models \neg b \land \neg d \land e \land f\\
%\mathbb{P}^c((a \mid b) {\land} \neg(e \mid f)), & \mbox{if } w \models \neg b \land \neg c \land d \land \neg f\\
%\mathbb{P}^c((c \mid d)), & \mbox{if } w \models a \land b \land \neg d \land e \land f\\
%\mathbb{P}^c(\neg(e \mid f)), &  \mbox{if } w \models a \land b \land \neg c \land d \\
%\mathbb{P}^c((c \mid d) {\lor} \neg(e \mid f)), & \mbox{if } w \models a \land b \land \neg d \land \neg f\\ \hdashline
%y, & \mbox{if } w \models \neg b \land \neg d \land \neg f\\
%\end{array}
%\right .
%$$
$$X_t(w) = \left \{
\begin{array}{ll}
1, & \mbox{if } w \models ab(cd \lor \bar{e}b) \\ % \land b \land ((c \land d) \lor (\neg e \land f)) \\
0, & \mbox{if } w \models \bar{a}b \lor \bar{c}def \\
\mathbb{P}^c((a \mid b)), & \mbox{if } w \models \bar{b} (c d \lor \bar{e}  f)\\
\mathbb{P}^c((a \mid b) {\land} (c \mid d)), & \mbox{if } w \models \bar{ b}  \bar{ d}   e  f\\
\mathbb{P}^c((a \mid b) {\land} \neg(e \mid f)), & \mbox{if } w \models \bar{ b}  \bar{c}  d  \bar{ f}\\
\mathbb{P}^c((c \mid d)), & \mbox{if } w \models a  b  \bar{d} e  f\\
\mathbb{P}^c(\neg(e \mid f)), &  \mbox{if } w \models a  b  \bar{c}  d \\
\mathbb{P}^c((c \mid d) {\lor} \neg(e \mid f)), & \mbox{if } w \models a  b  \bar{d} \bar{ f} \\[0.1cm] \hdashline  \\[-0.3cm] 
y, & \mbox{if } w \models \bar{ b}  \bar{ d} \bar{ f}\\
\end{array}
\right .
$$
where $y = \mathbb{P}^c((a \mid b) {\land} ((c \mid d) {\lor} \neg(e \mid f)))$, and  for simplicity, we have written $\mathbb{P}^c(s)$ for $\mathbb{P}^c(X_{s})$, i.e. for  $\mathbb{P}(X_{s}\mid {\bf b}(s))$. 
\end{example}

%Note that Definition \ref{xt}  is in fact a recursive definition, since $$\mathbb{P}(X_{t^w} \mid {\bf b}(t^w)) = \sum_{w'} X_{t^w}(w') \cdot P(w' \mid {\bf b}(t^w)),$$
%and in turn, 
%$$X_{t^w}(w') =\mathbb{P}(X_{(t^w)^{w'}} \mid {\bf b}((t^w)^{w'})),$$
%and so on, until reaching basic conditionals. 

We end this section with the following remark on a betting interpretation for compound conditionals $t \in \mathbb{T}(A)$. A bet between a gambler and a bookmaker on $X_t$  is specified as follows: %if $w \models b_1 \lor \ldots \lor b_n$ then: 
\begin{itemize}
\item The gambler pays to the bookmaker: $\mathbb{P}^c(X_t)$ 
\item  In a situation $w \in \Omega$, the gambler receives from the bookmaker: $X_t(w) = \mathbb{P}^c(X_{t^w})$
\end{itemize}
Note that if $w \models \neg {\bf b}$ then $t^w = t$, and thus $\mathbb{P}^c(X_{t^w}) = \mathbb{P}^c(X_t)$, i.e.\ what the gambler receives is what he has payed, i.e. it represents the situation in which the bet is called off. 
We observe that 
%in a bet on (the random quantity of) a compound conditional $t$, 
coherence requires that the prevision of the random gain must be equal to zero. Indeed, given $t$,  
$X_t$ is specified as: 
$$X_{t}(w) = 
%\mathbb{P}(X_{t^w} \mid {\bf b}(t^w)) = 
\left \{
\begin{array}{ll}
\mathbb{P}^c(X_{E_1}), & \mbox{if } w \models E_1 \\
\ldots &  \ldots \\
\mathbb{P}^c(X_{E_k}), & \mbox{if } w \models E_k \\[0.1cm] \hdashline  \\[-0.3cm]
\mathbb{P}^c(X_t), & \mbox{if } w \models \neg (E_1 \lor \ldots \lor E_k)\\
\end{array}
\right . .
$$
%where $Red^0(t) = \{ t^w \mid w \in \at({\bf A})\} = \{E_1, E_2, ..., E_k\}$. Then, 
In this setting, 
%it is clear that 
%$$\mathbb{P}^c(X_t) = \mathbb{P}^c(X_{E_1}) \cdot P(E_1 \mid  {\bf b}(t)) + ... +  \mathbb{P}^c(X_{E_k}) \cdot P(E_k \mid  {\bf b}(t)) .$$
the balance of the betting for the bookmaker in a given situation $w$ such that $w \models E_i$ is what he receives, $\prev(X_t)$, minus what he pays to the gambler, $X_t$, that is:
$$G(w) =
\mathbb{P}(X_t) - X_t(w)=
\mathbb{P}^c(X_t) - \mathbb{P}^c(X_{E_i}).$$
Moreover, when $w \models \no{{\bf b}}(t)$, it holds that  $G(w)=\mathbb{P}^c(X_t)-\mathbb{P}^c(X_t)=0.$ Then, 
the prevision of the random quantity $G$ of the balance will be: 
$$
\begin{array}{lll}
\mathbb{P}(G) &= & \sum_{w \in \Omega} G(w) \cdot P(w \mid {\bf b}(t)) = \\
 &= & \sum^{k}_{i=1}(\mathbb{P}^c(X_t) - \mathbb{P}^c(X_{E_i})) \cdot P(E_i \mid {\bf b}(t)) = \\  
 &= &  \mathbb{P}^c(X_t) \cdot  \sum^{k}_{i=1} P(E_i \mid {\bf b}(t))  -\\ 
 &= & \sum^{k}_{i=1} \mathbb{P}^c(X_{E_i})) \cdot P(E_i \mid {\bf b}(t))  = \\
 &= & \mathbb{P}^c(X_t) - \mathbb{P}^c(X_t) = 0 . 
\end{array}
$$

\section{Properties of compound conditionals}\label{secProperties}

In this section we show some key properties of the random quantities associated to certain forms of compound quantities that will allow us to show they can be endowed somehow with a Boolean algebraic structure

We start by showing  that if  two compound conditional terms $t$ and $t'$   are such that  ${\bf b}(t) \equiv {\bf b}(t')$,
to check whether they have the same associated random quantity, it is enough to check whether the random quantities take the same value over those worlds satisfying the disjunction ${\bf b}(t)$. A related result is given in  \cite[Theorem 4]{GiSa14}.

% We start by showing that if we have two compound conditional terms $t$ and $t'$ involving the same basic conditionals, that is such that  ${\bf b}(t) \equiv {\bf b}(t')$, to check whether they have the same associated random quantity, it is enough to check whether the random quantities take the same value over those worlds satisfying the disjunction ${\bf b}(t)$. 
%\begin{fact}\label{fact1} If $w \models \neg {\bf b}(t)$, then $X_t(w) =  \mathbb{P}(X_{t} \mid {\bf b}(t))$.
%\end{fact}
\begin{lemma} \label{equal} Let $t, t' \in \mathbb{T}(A)$ such that ${\bf b}(t) \equiv {\bf b}(t')$ and $X_t(w) = X_{t'}(w)$ for each $w \in \Omega)$ such that $w \models {\bf b}(t)$. Then $X_t = X_{t'}$. 
\end{lemma}

\begin{proof} It is enough to show that $X_t(w) = X_{t'}(w)$ when  $w \models \neg {\bf b}(t)$. Due to the fact that if $w \models \neg {\bf b}(t)$ then $X_t(w) =  \mathbb{P}(X_{t} \mid {\bf b}(t))$, this is equivalent to show that $\mathbb{P}(X_{t} \mid {\bf b}(t)) =   \mathbb{P}(X_{t'} \mid {\bf b}(t))$. But these previsions only depend on the values of $X_t$ and $X_{t'}$ on $w$'s such that  $w \models {\bf b}(t)$, and by hypothesis they coincide. 
\end{proof}

Clearly, if ${\bf b}(t) \equiv {\bf b}(t')$ and $t^w = t'^w$ then $X_t(w) = \mathbb{P}(X_{t^w} \mid {\bf b}(t)) = \mathbb{P}(X_{t'^w}  \mid {\bf b}(t)) = X_{t'}(w)$. Hence we have the following corollary. 

\begin{corollary} If ${\bf b}(t) = {\bf b}(t')$ and $t^w = t'^w$ for each $w \in \Omega$ such that $w \models {\bf b}(t)$, then $X_t = X_{t'}$. 
\end{corollary}

%\begin{example}\label{ex1} Let $t = (a \mid b) {\land} ((c \mid d) {\lor} \neg(e \mid f))$. Then, by applying the above definition, we get
%$$X_t(w) = \left \{
%\begin{array}{ll}
%1, & \mbox{if } w \models a \land b \land ((c \land d) \lor (\neg e \land f)) \\
%0, & \mbox{if } w \models (\neg a \land b) \lor (\neg c \land d \land e \land f) \\
%\mathbb{P}^c((a \mid b)), & \mbox{if } w \models \neg b \land ((c \land d) \lor (\neg e \land f))\\
%\mathbb{P}^c((a \mid b) {\land} (c \mid d)), & \mbox{if } w \models \neg b \land \neg d \land e \land f\\
%\mathbb{P}^c((a \mid b) {\land} \neg(e \mid f)), & \mbox{if } w \models \neg b \land \neg c \land d \land \neg f\\
%\mathbb{P}^c((c \mid d)), & \mbox{if } w \models a \land b \land \neg d \land e \land f\\
%\mathbb{P}^c(\neg(e \mid f)), &  \mbox{if } w \models a \land b \land \neg c \land d \\
%\mathbb{P}^c((c \mid d) {\lor} \neg(e \mid f)), & \mbox{if } w \models a \land b \land \neg d \land \neg f\\ \hdashline
%\mathbb{P}^c((a \mid b) {\land} ((c \mid d) {\lor} \neg(e \mid f))), & \mbox{if } w \models \neg b \land \neg d \land \neg f\\
%\end{array}
%\right .
%$$
%where, for simplicity, we have written $\mathbb{P}^c(s)$ for $\mathbb{P}^c(X_{s})$, i.e. for  $\mathbb{P}(X_{s}\mid {\bf b}(s))$. 
%\end{example}

\begin{proposition} \label{substi} Let $t \in \mathbb{T}(A)$ and let $s$ be a subterm of $t$. Further let $s' \in \mathbb{T}(A)$ such that ${\bf b}(s) = {\bf b}(s')$ and $X_{s^w} = X_{s'^w}$ for all $w \models {\bf b}(s)$, and let $t'$ the term obtained by uniformly replacing $s$ by $s'$ in $t$. Then $X_t = X_{t'}$. 
\end{proposition}

\begin{proof} It follows from the recursive definition of $X_t$ and  that, for any $w \in \Omega$, %the $w$-reduct of $t$, 
$t^w$ can be computed from the $w$-reducts of its subterms (Lemma \ref{reduct}). 
\end{proof}

%\begin{proposition}\label{propNeg} For every $t \in \mathbb{T}(A)$, %with $B = Cond(t)$, it holds that
%$X_{\neg t} = 1 - X_t$.%, and hence $P^*(\neg t) = 1 - P^*(t)$.
%\end{proposition}
%
%\begin{proof}
%By induction: 
%
%1) Base case $t = (a \mid b)$: for every $w$ satisfying ${\bf b}(t)$, we have $X_{\neg (a \mid b)}(w) = X_{(\neg a \mid b)}(w)$. Indeed: 
%$$X_{\neg (a \mid b)}(w) = \left \{
%\begin{array}{ll}
%\neg 1 = 0, & \mbox{if } w \models a \land b \\
%\neg 0 = 1, &  \mbox{if } w \models \neg a \land b 
%\end{array}
%\right . $$
%$$ 
%X_{(\neg a \mid b)}(w) = \left \{
%\begin{array}{ll}
% 0, & \mbox{if } w \models a \land b \\
%1, &  \mbox{if } w \models \neg a \land b 
%\end{array}
%\right .
%$$
%Therefore, by Lemma \ref{equal}, $X_{\neg (a \mid b)} = X_{(\neg a \mid b)}$. Moreover, we know that $\mathbb{P}^c((\neg a \mid b) = P(\neg a \mid b) = 1 -P(a \mid b)$, so that $X_{\neg (a \mid b)} = X_{(\neg a \mid b)} = 1 - X_{(a \mid b)}$. 
%  
%2) Induction step: assume $X_{\neg t'} = 1 - X_{t'}$ for every $t'$ less complex than $t$. Note that ${\bf b}(\neg t') = {\bf b}(t')$. Then, if $w \models {\bf b}(t)$ we have by definition,  
%$X_{\neg t}(w) = \mathbb{P}(X_{(\neg t)^w} \mid {\bf b}(t^w))$, and since we can assume that $t^w$ is strictly less complex than $t$, by the induction hypothesis, $\mathbb{P}(X_{(\neg t)^w} \mid {\bf b}(t^w)) =  \mathbb{P}(1-X_{t^w} \mid {\bf b}(t^w)) = 1 -  \mathbb{P}(X_{t^w} \mid {\bf b}(t^w)) = 1-X_{t}(w)$. 
%\end{proof}

%\begin{corollary}  For every $t \in \mathbb{T}(A)$, $X_{\neg \neg t} = X_t$. 
%\end{corollary}

Next we show that the random quantities for compound conditionals satisfy properties very familiar from a Boolean algebraic perspective. 

\begin{proposition}  \label{properties} For every $t,s,r \in \mathbb{T}(A)$ the following conditions hold:
\begin{enumerate}
\item $X_t = X_{t {\land} t}$
\item $X_{t {\land} s} = X_{s {\land} t}$
\item  $X_{t {\land} (s {\land} r)} = X_{(t {\land} s) {\land} r}$
\item $X_{t {\land} \neg t} = 0$
\item  $X_{\neg(t {\land} s)} = X_{\neg t {\lor} \neg s}$
\item $X_{t {\land} (s {\lor} r)} = X_{(t {\land} s) {\lor} (t {\land} r)}$
\item $X_{t {\lor} s} = X_t + X_s - X_{t {\land} s}$
\item $X_{\neg t} = 1 - X_t$ and hence $X_{\neg\neg t}=X_t$
\item If $a\leq b$, then  $X_{(a\mid b)\wedge (a\mid b\vee c)}=X_{a\mid b\vee c}\,$.
\end{enumerate}
\end{proposition}

\begin{proof} The detailed proof can be found in the Appendix. Items (1)-(8) are proved in a similar way by induction on the complexity of the terms involved, and where the base cases are those only involving basic conditionals. Many of these base cases are proved in previous examples. 
\end{proof}

Now we show that  compound conditionals  actually form a Boolean algebra. In order to do it  we first properly arrange the compound conditionals  in equivalence classes each of which contains terms from $\mathbb{T}(A)$ that provides the same random quantity under {\em any} given conditional probability. 

More precisely, recalling observation (ii) after Def. \ref{xt},
%\ref{remFirst} (1), 
let us define the binary relation $\equiv$ on $\mathbb{T}(A)$ as follows: for all $t, s\in \mathbb{T}(A)$,
$$
t\equiv s \mbox{ iff }X_t=X_{s}
$$
under any conditional probability $P$ on $A\times A'$

It is immediate to check that $\equiv$ is an equivalence relation. Thus, let $\mathbb{T}(A)$ to be the quotient $\mathbb{T}(A)/\mathord\equiv$. If we denote by $[t]$ the equivalence class of a generic term $t\in \mathbb{T}(A)$ under $\equiv$, define operations on $\mathbb{T}(A)$ as follows: for all $[t], [s]\in \mathbb{T}(A)$, $[t] {\land}^* [s] = [s {\land} t]$, $[t] {\lor}^* [s] = [s {\lor} t]$, $\neg^* [t] = [\neg t]$, $0 = [(\bot \mid \top)]$, $1 = [(\top \mid \top)]$.

By Proposition \ref{substi}, the above operations are well defined. Moreover, the following holds:

\begin{theorem}\label{main} For every finite Boolean algebra ${\bf A}$, the structure  $\mathcal{T}({\bf A}) = (\mathbb{T}(A)/\mathord\equiv, {\land}^*, {\lor}^*, \neg^*, 0, 1)$ is a Boolean algebra.
\end{theorem}

\begin{proof}
 By Proposition  \ref{properties}, the operations ${\land}^*, {\lor}^*, \neg^*$ and constants $0, 1$ satisfy all the required equations for $\mathcal{T}({\bf A})$ being a Boolean algebra. 
\end{proof}

\begin{remark}\label{remOrder}
Notice that, given $[t]$ and $[s]$ in $\TA$, $[t]\leq [s]$ in the lattice order of $\TA$ iff $[t]\wedge^*[s]=[t]$ iff $[t\wedge s]=[t]$ iff $X_{t\wedge s}=X_t$.
\end{remark}
Next proposition  shows that well-known properties of conditionals also hold in the setting of the present paper.
\begin{proposition}\label{propXX} The following properties hold in $\TA$:
\begin{itemize}
\item[(i)] $[(a \mid a)] = 1$
\item[(ii)] $[(a \mid b) \land (c \mid b)] = [(a \land c \mid b)]$
\item[(iii)] $[\neg(a \mid b)] = [(\bar a \mid b)]$
\item[(iv)] $[(a \land b \mid b)] = [(a \mid b)]$
\item[(v)] $[(a \mid b) \land (b \mid c)] = [(a \mid c)]$, if $a \leq b \leq c$
\item[(vi)] if $a\leq b$, then $[(a\mid b\vee c)]\leq [(a\mid b)]$ for all $c$.
\end{itemize}
\end{proposition}
\begin{proof} For each one of the equalities above, of the form $[t]=[s]$, we proved in previous examples, that $X_t=X_s$.  As for the last claim, apply Proposition \ref{properties} (9). Then the claim follows by the  definition of $\TA$.
\end{proof}

\section{Compound conditionals and their probability}\label{secProb}

Since $\TA$ is a Boolean algebra, we are allowed to define probabilities on it. In this section we will  show that the  previsions $\mathbb{P}(X_{t})$'s of the random quantities $X_t$'s determine in fact  a probability on $\TA$.

%such that, for every $[t]\in \TA$, its value coincides with  the
%For any compound conditional $t \in \mathbb{T}(A)$, we will define its probability-like value, denoted $P^*(t)$, as the 
%conditional prevision $\mathbb{P}^c(X_{t})$ of its associated random quantity $X_t$. We recall that from (\ref{EQ:PREVXT}) it holds that $\mathbb{P}^c(X_{t})=\mathbb{P}(X_{t})$.

\begin{definition} \label{pstar} Given a  conditional probability $P: A \times A' \to [0, 1]$,  we define the mapping $P^*:\mathbb{T}(A)\to[0,1]$ as follows: for every $t\in \mathbb{T}(A)$, %for any term, $t \in \mathbb{T}(A)$, we define its {\em probability value} as 
$P^*(t) =_{def}   \mathbb{P}(X_{t})$. In other words, based on (\ref{EQ:PREVXT})
$$
\begin{array}{l}
P^*(t) = \mathbb{P}(X_{t})=
\mathbb{P}^c(X_{t})=
\mathbb{P}(X_t \mid  {\bf b}(t)) =  \\ =\sum_{w} X_t(w) \cdot P(w\mid  {\bf b}(t)) = \sum_{w} \mathbb{P}^c(X_{t^w}) \cdot P(w\mid  {\bf b}(t)).
\end{array}
$$
\end{definition}
%{\color{red} Notice that $P^*(t)$ can actually be regarded as a probability value for the term $t\in \mathbb{T}(A)$. Indeed, for all $t'\in [t]$, $P^*(t)=P^*(t')$. For this reason we will henceforth speak about ``the probability value $P^*$ of a term $t$'' rather than referring to the elements of $\TA$.}
%\begin{fact}

From the above definition, it is cleat that if $t$ and $t'$ are terms such $t\equiv t'$, then $P^*(t) = P^*(t')$. 

\begin{proposition} Any given conditional probability $P: A \times A' \to [0, 1]$ fully determines the probabilities $P^*(t)$ for all compound conditional $t \in \mathbb{T}(A)$. 
\end{proposition}

\begin{proof} Let $t \in \mathbb{T}(A)$ such that $Cond(t) = \{(a_1 \mid b_1), \ldots,(a_n \mid b_n)\}$. We have seen that $P^*((a_i \mid b_i)) = P(a_i \mid b_i)$. Suppose that for every $s \in Red^0(t)$, there is a family of conditionals $C_s \subseteq A \times A'$ such that $P^*(s)$ is of the form $P^*(s) = f_s(\{P(c_i \mid d_i) : (c_i \mid d_i) \in C_s\})$ for some function $f_s$. Then, by definition, we have 
$$
\begin{array}{l}
P^*(t) = \prev(X_t)= \sum_{w\in W} \mathbb{P}^c(X_{t^w}) \cdot P(w \mid {\bf b}(t) ) =  \\ 
=\sum_{w\in W} P^*(t^w) \cdot P(w \mid {\bf b}(t)) = \\
=\sum_{w\in W} f_{t^w}(\{ P(c \mid d) : (c \mid d) \in C_{t^w} \}) \cdot P(w \mid {\bf b}(t)),
\end{array}
$$
 the latter expression clearly only depending on $P$. 
\end{proof}

%Furthermore, notice that by Proposition \ref{propNeg} and by definition of $P^*$, for all $t\in \mathbb{T}(A)$, $P^*(\neg t) = 1 - P^*(t)$.
%\end{fact}

%\begin{example} Let us consider the case of a basic conditional $t = (a \mid b)$. Then  ${\bf b}(t) = b$ and 
%$$X_{a \mid b}(w) = \left \{
%\begin{array}{ll}
%1, & \mbox{if } w \models a \land b \\
%0, & \mbox{if } w \models \neg a \land b \\ \hdashline
%\mathbb{P}^c((a \mid b)),  & \mbox{if } w \models \neg b
%\end{array}
%\right .
%$$
%In this case its holds that $P^*((a \mid b)) = P(a \mid b)$. Indeed, 
%$$P^*((a \mid b)) =_{def} \mathbb{P}(X_{a \mid b} \mid b) = 1 \cdot P(a \land b \mid b)  + 0 \cdot P(\neg a \land b \mid b) + \mathbb{P}^c((a \mid b)) \cdot P(\neg b \mid b) = P(a \land b \mid b) = P(a \mid b),$$ 
%since $P(\neg b \mid b) = 0$.
%\end{example}

Moreover, the following is a direct consequence of Proposition \ref{properties}.
\begin{corollary}  \label{propprob} For every $t, s, r \in \mathbb{T}(A)$ the following conditions hold:
\begin{enumerate}
\item $P^*(t) = P^*(t {\land} t)$
\item $P^*(t {\land} s) = P^*(s {\land} t)$
\item  $P^*(t {\land} (s {\land} r)) = P^*((t {\land} s) {\land} r)$
\item $P^*(t {\land} \neg t) = 0$.
\item  $P^*(\neg(t {\land} s)) = P^*(\neg t {\lor} \neg s)$
\item $P^*(t {\land} (s {\lor} r)) = P^*((t {\land} s) {\lor} (t {\land} r))$
\item $P^*(t) = 1$ if $X_t = 1$
\item $P^*(t) = 0$ if $X_t = 0$
\item $P^*(t {\lor} s) = P^*(t) + P^*(s) - P^*(t {\land} s)$
\item $P^*(\neg t) = 1 - P^*(t)$.
\end{enumerate}
\end{corollary}

Notice that the last four items of the above corollary together with the above remarked fact that $t\equiv t'$ implies $P^*(t)=P^*(t')$, show that $P^*$ naturally induces a probability on the algebra $\TA$ that we will still denote by the same symbol. 
\begin{corollary}
For every conditional probability $P$ on $A\times A'$ the map $P^*$ is a probability measure on $\TA$. 
\end{corollary}
%Furthermore, as we have checked in Example \ref{firstex}, for every basic conditional $(a\mid b)$ we have $P^*((a\mid b))=P(a\mid b)$, thus complying with the celebrated {\em Stalnaker thesis}. 
Furthermore, as we have checked in Example \ref{firstex}, for every basic conditional $(a\mid b)$ we have $P^*((a\mid b))=P(a\mid b)$, thus complying with {\em the Equation} \cite{edgington95}, or {\em Stalnaker's hypothesis}, 
see, e.g.,  \cite{douven11b,SGOP20}.
The next example shows how to compute the probability $P^*$ of a (more complex) compound conditional.
%As for compound conditionals the next example gives a hint on how to compute their probability $P^*$.
\begin{example} Continuing the above Example \ref{ex1} for $t = (a \mid b) {\land} ((c \mid d) {\lor} \neg(e \mid f))$,  we have ${\bf b}(t) = b \lor d \lor f$ and
%$$X_t(w) = \left \{
%\begin{array}{ll}
%1, & \mbox{if } w \models a  b  c  d \lor \bar e  f \\
%0, & \mbox{if } w \models \bar a  b \lor \bar c  d  e  %f \\
%\mathbb{P}^c(a \mid b), & \mbox{if } w \models \bar b  %(c  d \lor \bar e  f)\\
%\mathbb{P}^c((a \mid b) {\land} (c \mid d)), & \mbox{if %} w \models \bar b  \bar d  e  f\\
%\mathbb{P}^c((a \mid b) {\land} \neg(e \mid f)), & %\mbox{if } w \models \bar b  \bar c  d  \bar f\\
%\mathbb{P}^c(c \mid d), & \mbox{if } w \models a  b  %\bar d  e  f\\
%\mathbb{P}^c(\neg(e \mid f)), &  \mbox{if } w \models a % b  \bar c  d \\
%\mathbb{P}^c((c \mid d) {\lor} \neg(e \mid f)), & %\mbox{if } w \models a  b  \bar d  \bar f \\[0.1cm] %\hdashline \\[-0.2cm]
%y, & \mbox{if } w \models \bar b  \bar d  \bar f\\
%\end{array}
%\right .
%$$
%where $y = \mathbb{P}^c((a \mid b) {\land} ((c \mid d) {\lor} \neg(e \mid f)))$. 
using the above Definition \ref{pstar} we get:\\

\begin{tabular}{ll}
%$P^*(t)$\hspace*{-0.3cm} &$=_{def} \mathbb{P}(X_t\mid {\bf b}(t)) = $ \\ 
$P^*(t)$\hspace*{-0.3cm} &$= \mathbb{P}(X_t)=\mathbb{P}(X_t\mid {\bf b}(t)) = $ \\ 
& 
$=P( a  b  c  d \lor \bar e  f) \mid b \lor d \lor f) + $ \\ 
& $+\, P(a \mid b) \cdot  P(\bar b  (c  d \lor \bar e  f)\mid b \lor d \lor f) + $ \\
& $+\,P^*((a \mid b) {\land} (c \mid d)) \cdot P(\bar b  \bar d  e  f \mid b \lor d \lor f) +$ \\  
& $+\,P^*((a \mid b) {\land} \neg(e \mid f)) \cdot P(\bar b  \bar c  d  \bar f \mid b \lor d \lor f) +$ \\
& $+\,P(c \mid d) \cdot P(a  b  \bar d  e  f \mid b \lor d \lor f) + $ \\ 
& $+\,(1-P(e \mid f)) \cdot P(a b  \bar c  d \mid  b \lor d \lor f) + $ \\\
& $+\,P^*((c \mid d) {\lor} \neg(e \mid f)) \cdot P( a  b  \bar d  \bar  f \mid b \lor d \lor f).$ \\
\end{tabular}
\end{example}

\section{Boolean algebras of random variables and Boolean algebras of conditionals}\label{secComparison}
The above Theorem \ref{main}, tells us that the random quantities that represent  conditionals form a Boolean algebra. Actually, in \cite{FlGH20}, the authors presented a way to construct Boolean algebras of conditionals that also represent conditional statements, although following a different line of thoughts.

For the sake of clarity, let us briefly recall how a Boolean algebras of conditional $\CA$ is defined for every Boolean algebra ${\bf A}$. Consider the free Boolean algebra $\free(A|A)$ generated by the set $A|A$ of all pairs $(a, b)\in A\times A$ such that $b\neq \bot$, in the language $\land, \lor, \neg, \top^*$. Every pair $(a, b)\in A|A$ will be henceforth denoted by $(a\mid b)$. Then, one considers the following set of basic requirements the algebra $\CA$ should satisfy:
\vspace{.1cm}
 
\noindent(R1) For all $b\in A'$, the conditional $(b \mid b)$ will be the top element of $\mathcal{C}({\bf A})$ and $(\neg b \mid b)$ will be the bottom; 
\vspace{.1cm}

\noindent(R2) Given $b \in A'$, the set of conditionals $A \mid b = \{(a \mid b) : a \in A\}$ will be the domain of a Boolean subalgebra of $\mathcal{C}({\bf A})$, and in particular when $b = \top$, this subalgebra will be isomorphic to ${\bf A}$;
\vspace{.1cm}

\noindent(R3)  In a conditional $(a \mid b)$ we can replace the consequent $a$ by $a \land b$, that is,  
the conditionals $(a \mid b)$ and $(a \land b \mid b)$ represent the same element of $\mathcal{C}({\bf A})$; 
\vspace{.1cm}

\noindent(R4)  For all $a\in A$ and all $b,c\in A'$, if $a \leq b \leq c$, then the result of conjunctively combining the conditionals $(a \mid b)$ and $ (b \mid c)$ must yield the conditional $(a \mid c)$. 
\vspace{.1cm}

\noindent Notice that (R4)  encodes a sort of restricted chaining of conditionals and it is inspired by the chain rule of conditional probabilities: $P(a \mid b) \cdot P(b \mid c) = P(a \mid c)$ whenever $a \leq b \leq c$. 
\vspace{.1cm}

One then proceeds by considering the smallest congruence relation $\equiv_{\mathfrak{C}}$ on $\free(A|A)$ satisfying: 
\begin{itemize}
\item[](C1) $(b \mid b) \equiv_{\mathfrak{C}} \top^*$, for all $b\in A'$;
\item[](C2) $(a_1 \mid b)\land (a_2 \mid b) \equiv_{\mathfrak{C}} (a_1\wedge a_2 \mid b)$, \\ \hspace*{0.7cm} for all $a_1, a_2\in A$, $b\in A'$;
\item[](C3) ${\neg}(a \mid b) \equiv_{\mathfrak{C}} (\neg a \mid b)$, for all $a\in A$, $b\in A'$;
\item[](C4) $(a\wedge b \mid b) \equiv_{\mathfrak{C}} (a \mid b)$, for all $a\in A$, $b\in A'$;
\item[](C5) $(a\mid b) \land (b \mid c) \equiv_{\mathfrak{C}} (a \mid c)$, \\ \hspace*{0.7cm} for all $a\in A$, $b,c\in A'$ such that $a \leq b \leq c$.
\end{itemize}
Note that (C1)-(C5) faithfully account for the requirements R1-R4 where, in particular,  (C2) and (C3) account for R2. 
Finally, the algebra  $\mathcal{C}({\bf A})$ is formally defined as follows. 
\begin{definition}\label{def:BAC}
For every Boolean algebra ${\bf A}$, the {\em Boolean algebra of conditionals} of ${\bf A}$ is the quotient structure 
$
\mathcal{C}({\bf A})= \free(A|A)/_{\equiv_\mathfrak{C}}.
$
\end{definition}  

By definition, the algebra $\CA$ is finite whenever so is ${\bf A}$. In particular, if ${\bf A}$ is atomic with atoms $\alpha_1,\ldots, \alpha_n$, in \cite{FlGH20} it is shown that the atoms of $\CA$ are all in the following form: let $\langle\beta_1,\ldots, \beta_{n-1}\rangle$ be a sequence of pairwise different atoms of ${\bf A}$. Then, $(\beta_1\mid \top)\wedge(\beta_2\mid  \overline{\beta_1})\wedge (\beta_{n-1}\mid \overline{\beta_1}\wedge\ldots\wedge \overline{\beta_{n-2}})$ is an atom of $\CA$ and indeed all atoms of $\CA$ have that form for some sequence  of atoms of ${\bf A}$ of length $n-1$. For any such sequence $\langle \alpha\rangle$, we will write $\omega_{\langle \alpha\rangle}$ to denote its associated atom of $\CA$.

For every sequence $\langle\alpha\rangle$, $\omega_{\langle\alpha\rangle}$ clearly belongs to $\mathbb{T}({\bf A})$. 
Thus it makes sense to look at these terms inside $\TA$. 
%to compute its probability by $P^*$. Next result shows that $P^*$ factorizes on the conjunctive components of every atom of $\CA$.

%For every sequence $\langle\alpha\rangle$, $\omega_{\langle\alpha\rangle}$ clearly belongs to $\mathbb{T}({\bf A})$. Thus it makes sense to compute its probability by $P^*$. 

%Next we show that probability of the conjunctions of conditonals defining the atoms in the algebra $C({\bf A})$ factorise here as well. 

\begin{theorem}\label{thmAtomsTA}
For every finite Boolean algebra ${\bf A}$ with atoms $\alpha_1,\ldots, \alpha_n$, the set $\at(\TA)$ of atoms of $\TA$ coincides with the set $\{[\omega_{\langle\alpha\rangle}]\in \TA\mid \omega_{\langle\alpha\rangle}$ is 
%a term of  $\mathbb{T}(A)$ defining 
an atom of  $\CA\}$.
\end{theorem}

\begin{proof}
Direct inspection on the proofs of Proposition 4.3 and Theorem 4.4 from \cite{FlGH20} where the authors proved that $\at(\CA)$ is the set of atoms of $\CA$, shows that the unique properties of $\CA$ needed in their proofs are the basic properties of Boolean algebras (in particular commutativity and associativity of $\wedge$ and distributivity), plus the following:
\begin{itemize}
\item If $a\leq b$, then for all $c$, $[(a\mid b\vee c)]\leq [(a\mid b)]$ (see part (ii) in the proof of \cite[Theorem 4.4]{FlGH20});
\item for all $a\neq \bot$, $[(a\mid a)]=1$ and $[(\bot\mid a)]=0$ (see part (b) in  the proof of \cite[Proposition 4.3]{FlGH20});
\item $[(a\mid b)]=\bigvee_{\alpha \in \at({\bf A}): \alpha \leq a}[(\alpha\mid b)]$; (see (a) in the proof of \cite[Proposition 4.3]{FlGH20}).
\end{itemize}
All these properties hold in $\TA$, whence the same proofs applies to this case.
%By Lemma \ref{lemmaPartition},  $X$ is a partition of $\TA$. Let $\omega=\omega_{\langle \alpha\rangle}$ where $\langle\alpha\rangle=\langle \alpha_1,\ldots, \alpha_2,\ldots,\alpha_{n-1}$. Then it is left to prove that, if $[\omega]\leq [t]\leq [\bot]$, then either $[\omega]=[t]$ or $[t]=[\bot]$. Let us start assuming that $t=(\alpha_k\mid b)$ with $\alpha_k\leq b$. Since $\alpha_k$ is an atom of ${\bf A}$, there are two cases: 
%
%(1) If $b=\alpha_k\vee\ldots\vee \alpha_n$. In this case $\omega\wedge (\alpha_k\mid b)=\omega$ because of Proposition \ref{properties}.
%
%(2) If $b=\alpha_k\vee \alpha_i\vee c$ for some $i<k$
\end{proof}

\begin{corollary}\label{mainSec5}
For every finite Boolean algebra ${\bf A}$, 
$\mathcal{T}({\bf A})$ is  isomorphic to $\mathcal{C}({\bf A})$.
\end{corollary}
\begin{proof} Two Boolean algebras with the same cardinality are isomorphic.
\end{proof}

%\blue{***And we should say that $P^* = \mu_P$!!***}
In \cite{FlGH20} it is proved that any (unconditional) positive probability $P$ on ${\bf A}$ canonically extends to a  positive probability $\mu_P$ on $\CA$ such that for every basic conditional $(a\mid b)$, $\mu_P(a\mid b)=P(a\wedge b)/P(b)$. %In the recent \cite{FGGS} we extend that result to the case of the original $P$ being a conditional probability on ${\bf A}$. The proof can be found in the Appendix.

 We can finally apply the latter result and the above Theorem \ref{thmAtomsTA} to show our final outcome, for which we still need a previous lemma on the factorization of the probability $P^*$ on the atoms of $\TA$, that is in fact a particular case of \cite[Theorem 18]{GiSa20}. For the reader's convenience, we  add its proof in the Appendix.
 
\begin{lemma} \label{canonical} 
Let ${\bf A}$ be a finite Boolean algebra with $\at({\bf A}) = \{\alpha_1, \ldots, \alpha_n\}$ and let $P$ be a conditional probability on $A\times A'$. Then, 

\begin{center}
$P^*((\alpha_1 \mid \top) {\land} (\alpha_2 \mid \neg \alpha_1) {\land} \ldots {\land} (\alpha_{n-1} \mid \alpha_{n-1} \lor \alpha_n)) = $ \\
$=P(\alpha_1) \cdot P (\alpha_2 \mid \neg \alpha_1) \cdot \ldots \cdot P(\alpha_{n-1} \mid \alpha_{n-1} \lor \alpha_n)$.
\end{center}
\end{lemma}

We observe that, if $P$ is a positive probability on ${\bf A}$, for each basic conditional $(a\mid b)$, $P^*(a\mid b)=P(a\wedge b)/P(b)=\mu_P(a\mid b)$. Our last result shows that $P^*$ and $\mu_P$ coincide on every compound conditional as well.
 %(a\mid b)$. Finally we finally shows that $P^*$ is essentially the canonical extension $\mu_P$.
%apply the latter result and the above Corollary \ref{mainSec5} to show our final result.
 
\begin{theorem}\label{thmFinal} Given a positive probability $P$ on ${\bf A}$ and any $t\in \mathbb{T}(A)$, $P^*(t)=\mu_P(t)$, once we identify the elements of $\TA$ with those of $\CA$.
\end{theorem}

\begin{proof} By Lemma \ref{canonical}, $P^*$ coincides with $\mu_P$ on the atoms of $\mathcal{T}({\bf A})$, and hence on the whole algebra. 
\end{proof}

\section{Conclusions}\label{SecConclusion}
%Lluis:

In the present paper we have put forward an investigation on compound conditionals that, starting from the original setting proposed by de Finetti, aims at representing them in terms of conditional random quantities. Technically speaking, we start by a finite Boolean algebra ${\bf A}$ of events and a (coherent) conditional probability $P$ on $A\times A'$, where $A'=A\setminus\{\bot\}$. Then, to each term $t$ written in the language having as variables  basic conditionals of the form $(a\mid b)$ (for $a\in A$ and $b\in A'$),  we first consider,  for each interpretation $w\in \Omega$ the reduct $t^w$ of $t$, and then we associate to $t$ a conditional random quantity $X_t:\Omega\to[0,1]$, that assigns to each $w\in \Omega$ the value $X_t(w)$ given by the conditional prevision $\mathbb{P}(X_{t^w} \mid {\bf b}(t^w))$. 

By doing this, we have presented a natural and uniform procedure to interpret compound conditionals as random quantities and we have investigated the numerical and logical properties of such representation for compound conditionals via their associated random quantities. Our main contribution concerns  the possibility of defining operations among conditionals by an iterative procedure. Furthermore, we have proved that these operations allow us to regard the set of those numerical representations as a Boolean algebra $\TA$. This latter result provides in turn a numerical counterpart of the construction explored in \cite{FlGH20}, where the authors showed that compound symbolic conditionals, satisfying certain identities, can be endowed with a suitable Boolean algebra structure $\CA$. In particular, we have shown that the two algebraic structures that arise from these numerical and the symbolic representations of conditionals turn out to be isomorphic. Moreover, any conditional probability $P$ on $A\times A'$ extends in the same way to both algebras $\TA$ and $\CA$.

As for future work we plan to generalize Theorem \ref{thmFinal} to the case in which $P$ is a conditional probability on $A\times A'$. Moreover, we aim at studying the extension of the random quantity-based approach to compound conditionals developed in this paper to deal with general forms of iterated conditionals, see e.g. \cite{SGOP20,GiSa21E}. Finally, we plan to study possible applications of compound conditionals to non-monotonic reasoning from conditional bases and to conditional logics more in general,  
%in particular to different forms of probabilistic entailment 
in the line of possible interrelationships with different areas as discussed in \cite{Dagstuhl19}. Another interesting area of application of compound conditionals to be further explored is the area of  psychology of uncertain reasoning \cite{OVER20}.

\subsection*{Acknowledgments}  Flaminio and Godo acknowledge partial support by the MOSAIC project (EU H2020-MSCA-RISE-2020 Project 101007627) and also  by the Spanish project PID2019-111544GB-C21, funded by  MCIN/AEI/10.13039/501100011033. Sanfilippo acknowledges support by the FFR2020-FFR2021 projects of University of Palermo, Italy.

\section*{Appendix}

\subsection*{A1. Axiomatic conditional probability}
Given an algebra  of  events ${\bf A}$ and a non empty subfamily
 $B$ of $A'$, where  $A'=A\setminus \{\bot\}$, a (finitely-additive) \emph{conditional  probability}
   on $A\times B$ is a real-valued function  $P$ defined on $A\times B$ satisfying the following   properties, see, e.g., \cite{Renyi55,Csaszar55,Popper59,Dubins75,Holz85,Rega85,Cole94,coletti02}:
\renewcommand{\labelenumii}{\alph{enumi}}

 \begin{enumerate}
 \item   $P(\cdot|b)$ is a finitely additive probability on ${\bf A}$, for every $b\in B$;
 \item $P(b|b)$=1, for every $b\in B$;
 \item $P(ab|c) = P(a|bc)P(b|c)$,  for every $a\in A$ and  $b,c\in B$.
% \item[iii] (oppure) $P(E_1E_2|H) = P(E_1|E_2H)P(E_2|H)$, \\ for every $E_1,E_2\in \E$ and $H,E_2H\in \X$.
 \end{enumerate}
Notice that, when $c=\top$,  from 3 it follows that $P(ab)=P(a|b)P(b)$, where $P(\cdot)=P(\cdot|\top)$. In  particular, $P(a|b)=P(ab)/P(b)$, when $P(b)>0$.
 
\subsection*{A2. Proof of Propostion 4.4} 

\noindent{\bf Proposition 4.4.}  {\em  For every $t,s,r \in \mathbb{T}(A)$ the following conditions hold:
\begin{enumerate}
\item $X_t = X_{t {\land} t}$
\item $X_{t {\land} s} = X_{s {\land} t}$
\item  $X_{t {\land} (s {\land} r)} = X_{(t {\land} s) {\land} r}$
\item $X_{t {\land} \neg t} = 0$.
\item  $X_{\neg(t {\land} s)} = X_{\neg t {\lor} \neg s}$
\item $X_{t {\lor} (s {\land} r)} = X_{(t {\land} s) {\lor} (t {\lor} r)}$
\item $X_{t {\lor} s} = X_t + X_s - X_{t {\land} s}$
\item $X_{\neg t} = 1 - X_t$ and hence $X_{\neg\neg t}=X_t$
\item If $a\leq b$, then  $X_{(a\mid b)\wedge (a\mid b\vee c)}=X_{a\mid b\vee c}$.
\end{enumerate}
}

\begin{proof}
The proofs of the items (1)-(8) are by induction on the complexity of the terms $t,s,r$ appearing in the expressions. We will show the proofs of (1), (2), (3), (4),  (8) and the basic case of (7) to exemplify the general idea. Also notice that, by Lemma 4.1, to prove  an identity $X_t=X_s$ where ${\bf b}(t)={\bf b}(s)$ it is enough to check that $X_t(w)=X_s(w)$ for all $w\in {\bf b}(t)$. \vspace{.1cm}

\noindent (1) $X_t = X_{t {\land} t}$. 
Note that ${\bf b}(t) = {\bf b}(t {\land} t)$. We proceed by induction on the complexity of $t$: 
\begin{itemize}
\item Let $t = (a \mid b)$, then one can check that $(a \mid b)^w = ((a \mid b){\land}(a \mid b))^w \in  \{1, 0\}$ for every $w \models b$. Hence $X_{(a \mid b)} = X_{(a \mid b) {\land} (a \mid b)}$

%- $t = s {\land} r$, such that $X_s = X_{s{\land} s}$ and $X_r = X_{r{\land} r}$
\item Assume $X_{s {\land} s} = X_{s}$  for every $s \in Red^0(t)$ ands we check that $X_{t^w} = X_{(t {\land} t)^w}$ for every $w \in \at({\bf A})$ such that $w \models {\bf b}(t)$. But $t^w \in Red^0(t)$, hence by the inductive hypothesis, we have $X_{(t {\land} t)^w} =  X_{t^w {\land} t^w} = X_{t^w}$, hence $X_t = X_{t {\land} t}$. 
\end{itemize}

\noindent (2) $X_{t {\land} s} = X_{s {\land} t}$. 

\begin{itemize}
\item First we check the condition for basic conditionals $t = (a \mid b)$ and $s = (c \mid d)$ by computing their conjunctions as in Example 3.8.
\item Then, assume $X_{t' {\land} s'} = X_{s' {\land} t'}$ for every $t', s'$ such that $t' {\land} s' \in Red^0(t {\land} s)$. Then, again use Lemma 4.1 and show $X_{(t {\land} s)^w} = X_{(s {\land} t)^w}$ for every $w \in {\bf b}(t {\land} s) = {\bf b}(s {\land} t)$, but $X_{(t {\land} s)^w} = X_{t^w {\land} s^w} =  X_{s^w {\land} t^w}  = X_{(s {\land} t)^w}$.  
\end{itemize}

\noindent (3)  $X_{t {\land} (s {\land} r)} = X_{(t {\land} s) {\land} r}$. The proof is analogous to the previous one: 
\begin{itemize}
\item First check that the condition for basic conditionals $t = (a \mid b)$, $s = (c \mid d)$ and $r = (e \mid f)$. Indeed, $((a \mid b) \land ( (c \mid d) \land (e \mid f)))^w = 
(((a \mid b) \land (c \mid d)) \land (e \mid f))^w \in \{(a \mid b) \land (c \mid d), (a \mid b) \land (e \mid f),(c \mid d) \land (e \mid f), (a \mid b), (c \mid d), (e \mid f), 1, 0\}$ for every $w \models b {\lor} d {\lor} f$. Hence, $X_{(a \mid b) \land ( (c \mid d) \land (e \mid f))} = X_{((a \mid b) \land (c \mid d)) \land (e \mid f)}$.

\item Then, assume  $X_{t' {\land} (s' {\land} r')} = X_{(t' {\land} s') {\land} r'}$ for every $t' \land ( s' \land r') \in Red^0(t {\land} (s{\land} r))$. Then, again use Lemma 4.1 and show $X_{(t {\land} (s {\land} r))^w} = X_{((t {\land} s) {\land} t)^w}$ for every $w \in {\bf b}(t {\land} (s {\land} r)) = {\bf b}((t {\land} s) {\land} t)$.  But $X_{(t {\land} (s {\land} r))^w} = X_{t^w {\land} (s^w {\land} r^w)} =  X_{(t^w {\land} s^w ){\land} r^w}  = X_{(t {\land} s) {\land} r)^w}$.  
\end{itemize}

%\vspace{.1cm}

\noindent (4) $X_{t {\land} \neg t} = 0$. Clearly, $X_{(a \mid b) {\land} \neg (a \mid b)}(w) = 0$ for every $w \models b$. If $t$ is arbitrary, by the inductive hypothesis, we have $X_{(t {\land} \neg t)^w} =  X_{t^w {\land} \neg t^w} = 0$, and thus  $X_{t {\land} \neg t} = 0$ as well.

\vspace{.1cm}

%\noindent (5)  $X_{\neg(t {\land} s)} = X_{\neg t {\lor} \neg s}$

%- $t = (a \mid b), s = (c \mid d)$

%- Suppose $X_{\neg(t' {\land} s')} = X_{\neg t' {\lor} \neg s'}$ for any $t', s'$ such that $t' {\land} s' \in Red^0(t {\land} s)$. 

 %$X_{(\neg(t {\land} s))^w}= X_{\neg(t^w {\land} s^w)} = X_{\neg t^w {\lor} \neg s^w} = X_{(\neg t {\lor} \neg s)^w} $
 
% \vspace{.1cm}

%\noindent (6) $X_{t {\land} (s {\lor} r)} = X_{(t {\lor} s) {\land} (t {\lor} r)}$

%Analogously

%\vspace{.1cm}

\noindent (7) $X_{t {\lor} s} = X_t + X_s - X_{t {\land} s}$. 
As usual, first we check the condition for basic conditionals $t = (a \mid b)$ and $s = (c \mid d)$. Recalling how $X_{(a \mid b) {\land} (c \mid d)}(w)$ is defined in Example 3.8, 
a simple check gives;
\begin{center}
$X_{(a \mid b)}(w) + X_{(c \mid d)}(w) - X_{(a \mid b) {\land} (c \mid d)}(w) = $ \\  \mbox{} \\
$\left \{
\begin{array}{ll}
1, & \mbox{if } w \models a  b \lor c  d \\
0, & \mbox{if } w \models \bar a  b  \bar c  d  \\
{P}(a \mid b), & \mbox{if } w \models \bar b \bar c  d\\
{P}(c \mid d), & \mbox{if } w \models \bar a  b  \bar d \\[0.1cm] \hdashline  \\[-0.3cm]
{P}(a \mid b) + {P}(c \mid d) - \\ %& \mbox{if } w \models \neg b \land \neg d \\
\mathbb{P}^c((a \mid b) {\land} (c \mid d)), & \mbox{if } w \models \bar b  \bar d \\
\end{array}
\right .
$
\end{center}
while:
$$X_{(a \mid b) {\lor} (c \mid d)}(w) = \left \{
\begin{array}{ll}
1, & \mbox{if } w \models a  b \lor c  d \\
0, & \mbox{if } w \models \bar a  b \bar c d \\
{P}(a \mid b), & \mbox{if } w \models \bar b  \bar c  d\\
{P}(c \mid d), & \mbox{if } w \models \bar a  b  \bar d \\[0.1cm] \hdashline  \\[-0.3cm]
\mathbb{P}^c((a \mid b) {\lor} (c \mid d)), & \mbox{if } w \models \bar b  \bar d \\
\end{array}
\right .
$$
Therefore  $X_{(a \mid b)} + X_{(c \mid d)} - X_{(a \mid b) {\land} (c \mid d)} = X_{(a \mid b) {\lor} (c \mid d)}$ over $b \lor d$, then it coincides over the whole $\at({\bf A})$.

 \vspace{.1cm}
 
\noindent 
(8) Again by induction on $t$. Notice that the case $t=(a\mid b)$ has been already proved in Example 3.7. Then, assume $X_{\neg t'} = 1 - X_{t'}$ for every $t'$ less complex than $t$. Note that ${\bf b}(\neg t') = {\bf b}(t')$. Then, if $w \models {\bf b}(t)$, by definition,  
$X_{\neg t}(w) = \mathbb{P}(X_{(\neg t)^w} \mid {\bf b}(t^w))$, and since we can assume that $t^w$ is strictly less complex than $t$, by the induction hypothesis, $\mathbb{P}(X_{(\neg t)^w} \mid {\bf b}(t^w)) =  \mathbb{P}(1-X_{t^w} \mid {\bf b}(t^w)) = 1 -  \mathbb{P}(X_{t^w} \mid {\bf b}(t^w)) = 1-X_{t}(w)$. 

 \vspace{.1cm}
 
\noindent
(9) The claim follows by direct computation of the random quantities $X_{(a\mid b)\wedge (a\mid b\vee c)}$ and of $X_{(a\mid b\vee c)}$, when $a \leq b$, by using the generic definition  of the random quantity $X_{t\wedge s}$ for any $t$ and $s$ as given in Definition 3.5.
% and moreover, $\mathbb{P}^c((a \mid b) {\lor} (c \mid d) = {P}(a \mid b) + {P}(c \mid d) -\mathbb{P}^c((a \mid b)
%\end{itemize}
\end{proof}

\subsection*{A3. Proof of Lemma 6.3} 
 
\noindent {\bf Lemma 6.3.} {\em 
Let ${\bf A}$ be a finite Boolean algebra with $\at({\bf A}) = \{\alpha_1, \ldots, \alpha_n\}$ and let $P$ be a conditional probability on $A\times A'$. Then, 

\begin{center}
$P^*((\alpha_1 \mid \top) {\land} (\alpha_2 \mid \neg \alpha_1) {\land} \ldots {\land} (\alpha_{n-1} \mid \alpha_{n-1} \lor \alpha_n)) = $ \\
$P(\alpha_1) \cdot P (\alpha_2 \mid \neg \alpha_1) \cdot \ldots \cdot P(\alpha_{n-1} \mid \alpha_{n-1} \lor \alpha_n)$.
\end{center}
}

\begin{proof} Let $A_i = (\alpha_i \mid \alpha_i {\lor} \ldots\vee \alpha_{n}) {\land} \ldots {\land} (\alpha_{n-1} \mid \alpha_{n-1} {\lor} \alpha_n)$. We will prove first that, for every $1 \leq i \leq n-1$, the following equality holds: 
$$
\begin{array}{ll} 
P^*(A_i) = 
P(\alpha_i \mid  \alpha_i {\lor} ... {\lor} \alpha_n) \cdot P^*(A_{i+1})
%\\ P(\alpha_{i+1} \mid \alpha_{i+1} \lor  ... \lor \alpha_n) \cdot \ldots 
%\cdot P(\alpha_{n-1} \mid \alpha_{n-1} \lor \alpha_n)). 
\end{array}
$$
Indeed, one can check that the random quantity $X_{A_i}$ is defined as follows:
$$X_{A_i}(w) =
\left \{ \begin{array}{ll}
 P^*(A_i),  &   \mbox{if } w \models \alpha_1 {\lor} \cdots {\lor} \alpha_{i-1}, \\
 P^*(A_{i+1}), & \mbox{if } w \models \alpha_i, \\
 0,  & \mbox{if } w \models \alpha_{i+1}\vee \cdots \vee \alpha_n .
\end{array} \right .
$$
Hence, since ${\bf b}(A_i) = \alpha_i \lor \ldots \lor \alpha_n$, we have 
\[
P^*(A_i)=\prev^c(X_{A_i}) = P^*(A_{i+1}) \cdot P(\alpha_i \mid \alpha_i \lor \ldots \lor \alpha_n).
\]
Now, by repeatedly applying this equality, we finally have: 
$$
\begin{array}{ll}
P^*(A_1) =  
P(\alpha_1) \cdot P^*(A_2) = \\=  P(\alpha_1) \cdot P (\alpha_2 \mid \neg \alpha_1) \cdot P^*(A_3) = \cdots = \\= P(\alpha_1) \cdot P (\alpha_2 \mid \neg \alpha_1) \cdot \ldots \cdot P(\alpha_{n-1} \mid \alpha_{n-1} \lor \alpha_n).
\end{array}
$$
Then the proposition follows just by observing that $A_1 = (\alpha_1 \mid \top) {\land} (\alpha_2 \mid \neg \alpha_1) {\land} \ldots {\land} (\alpha_{n-1} \mid \alpha_{n-1} \lor \alpha_n)$. 
\end{proof}

%\bibliography{biblio}

\begin{thebibliography}{10}
	\providecommand{\url}[1]{\texttt{#1}}
	\providecommand{\urlprefix}{URL }
	\providecommand{\doi}[1]{https://doi.org/#1}
	
	\bibitem{adams75}
	Adams, E.W.: The logic of conditionals. Reidel, Dordrecht (1975)
	
	\bibitem{Dagstuhl19}
	Aucher, G., {\'{E}}gr{\'{e}}, P., Kern{-}Isberner, G., Poggiolesi, F.:
	Conditional logics and conditional reasoning: New joint perspectives
	({D}agstuhl seminar 19032). Dagstuhl Reports  \textbf{9}(1),  47--66 (2019).
	\doi{10.4230/DagRep.9.1.47}, \url{https://doi.org/10.4230/DagRep.9.1.47}
	
	\bibitem{BeierleEKK18}
	Beierle, C., Eichhorn, C., Kern{-}Isberner, G., Kutsch, S.: Properties of
	skeptical c-inference for conditional knowledge bases and its realization as
	a constraint satisfaction problem. Ann. Math. Artif. Intell.
	\textbf{83}(3-4),  247--275 (2018)
	
	\bibitem{benferhat97}
	Benferhat, S., Dubois, D., Prade, H.: Nonmonotonic reasoning, conditional
	objects and possibility theory. Artificial Intelligence  \textbf{92},
	259--276 (1997). \doi{10.1016/S0004-3702(97)00012-X}
	
	\bibitem{Cala87}
	Calabrese, P.: An algebraic synthesis of the foundations of logic and
	probability. Information Sciences  \textbf{42}(3),  187 -- 237 (1987).
	\doi{10.1016/0020-0255(87)90023-5}
	
	\bibitem{CiDu13}
	Ciucci, D., Dubois, D.: A map of dependencies among three-valued logics.
	Information Sciences  \textbf{250},  162 -- 177 (2013).
	\doi{10.1016/j.ins.2013.06.040}
	
	\bibitem{Cole94}
	Coletti, G.: Coherent numerical and ordinal probabilistic assessments. IEEE
	Trans. on Systems, Man, and Cybernetics  \textbf{24}(12),  1747--1754 (1994)
	
	\bibitem{coletti02}
	Coletti, G., Scozzafava, R.: Probabilistic logic in a coherent setting. Kluwer,
	Dordrecht (2002)
	
	\bibitem{ColettiPV19}
	Coletti, G., Petturiti, D., Vantaggi, B.: Dutch book rationality conditions for
	conditional preferences under ambiguity. Ann. Oper. Res.  \textbf{279}(1-2),
	115--150 (2019)
	
	\bibitem{Csaszar55}
	Cs{\'a}sz{\'a}r, A.: Sur la structure des espace de probabilit{\'e}
	conditionelle. Acta Mathematica Academiae Scientiarum Hungaricae  \textbf{6},
	337--361 (1955)
	
	\bibitem{douven11b}
	Douven, I., Dietz, R.: A puzzle about {S}talnaker's hypothesis. Topoi pp.
	31--37 (2011). \doi{10.1007/s11245-010-9082-3}
	
	\bibitem{Dubins75}
	Dubins, L.E.: Finitely additive conditional probabilities, conglomerability and
	disintegrations. The Annals of Probability  \textbf{3},  89--99 (1975)
	
	\bibitem{DuboisP94}
	Dubois, D., Prade, H.: Conditional objects as nonmonotonic consequence
	relations: Main results. In: Doyle, J., Sandewall, E., Torasso, P. (eds.)
	Proceedings of the 4th International Conference on Principles of Knowledge
	Representation and Reasoning (KR'94). Bonn, Germany, May 24-27, 1994. pp.
	170--177. Morgan Kaufmann (1994)
	
	\bibitem{edgington95}
	Edgington, D.: On conditionals. Mind  \textbf{104},  235--329 (1995)
	
	\bibitem{OVER20}
	Elqayam, S., Douven, I., Evans, J.S.B.T., Cruz, N.: Logic and Uncertainty in
	the Human Mind: A Tribute to David E. Over. Routledge, Oxon (2020).
	\doi{10.4324/9781315111902}
	
	\bibitem{deFi35}
	de~Finetti, B.: {La Logique de la Probabilit\'{e}}. In: Actes du Congr\`{e}s
	International de Philosophie Scientifique, Paris, 1935, pp. IV 1 -- IV 9.
	Hermann et C$^{ie}$ \'Editeurs, Paris (1936)
	
	\bibitem{definetti37-2}
	de~Finetti, B.: La pr{\'e}vision: ses lois logiques, ses sources subjectives.
	Annales de l'{I}nstitut Henri Poincar{\'e}  \textbf{7}(1),  1--68 (1937)
	
	\bibitem{FlGH20}
	Flaminio, T., Godo, L., Hosni, H.: Boolean algebras of conditionals,
	probability and logic. Artificial Intelligence  \textbf{286},  103347 (2020).
	\doi{10.1016/j.artint.2020.103347}
	
	\bibitem{vanFraassen1976261}
	van Fraassen, B.: Probabilities of conditionals. Foundations of Probability
	Theory, Statistical Inference, and Statistical Theories of Science
	\textbf{1},  261--308 (1976), cited By 89
	
	\bibitem{Ghirardato}
	Ghirardato, P.: Revisiting savage in a conditional world. Economic Theory
	\textbf{20}(1),  83--92 (2002)
	
	\bibitem{gili89}
	Gilio, A.: Classi quasi additive di eventi e coerenza di probabilit{\`a}
	condizionate. Rendiconti dell'Istituto di Matematica dell'Universit{\`a} di
	Trieste  \textbf{XXI}(1),  22--38 (1989)
	
	\bibitem{Gili95a}
	Gilio, A.: Algorithms for precise and imprecise conditional probability
	assessments. Mathematical Models for Handling Partial Knowledge in Artificial
	Intelligence pp. 231--254 (1995)
	
	\bibitem{gilio02}
	Gilio, A.: Probabilistic reasoning under coherence in {S}ystem {P}. Annals of
	Mathematics and Artificial Intelligence  \textbf{34},  5--34 (2002).
	\doi{10.1023/A:101442261}
	
	\bibitem{GiSa13a}
	Gilio, A., Sanfilippo, G.: {C}onjunction, disjunction and iterated conditioning
	of conditional events. In: Synergies of Soft Computing and Statistics for
	Intelligent Data Analysis, AISC, vol.~190, pp. 399--407. Springer, Berlin
	(2013). \doi{10.1007/978-3-642-33042-1\_43}
	
	\bibitem{gilio13}
	Gilio, A., Sanfilippo, G.: Quasi conjunction, quasi disjunction, t-norms and
	t-conorms: {P}robabilistic aspects. Information Sciences  \textbf{245},
	146--167 (2013). \doi{10.1016/j.ins.2013.03.019}
	
	\bibitem{GiSa14}
	Gilio, A., Sanfilippo, G.: Conditional random quantities and compounds of
	conditionals. Studia Logica  \textbf{102}(4),  709--729 (2014).
	\doi{10.1007/s11225-013-9511-6}
	
	\bibitem{GiSa19}
	Gilio, A., Sanfilippo, G.: Generalized logical operations among conditional
	events. Applied Intelligence  \textbf{49}(1),  79--102 (Jan 2019).
	\doi{10.1007/s10489-018-1229-8}
	
	\bibitem{GiSa20}
	Gilio, A., Sanfilippo, G.: Algebraic aspects and coherence conditions for
	conjoined and disjoined conditionals. International Journal of Approximate
	Reasoning  \textbf{126},  98 -- 123 (2020)
	
	\bibitem{GiSa21A}
	Gilio, A., Sanfilippo, G.: On compound and iterated conditionals. Argumenta
	\textbf{6}(2),  241--266 (2021). \doi{10.14275/2465-2334/202112.gil},
	\url{https://www.argumenta.org/article/compound-iterated-conditionals/}
	
	\bibitem{GiSp92}
	Gilio, A., Spezzaferri, F.: Knowledge integration for conditional probability
	assessments. In: Dubois, D., Wellman, M.P., D'Ambrosio, B., Smets, P. (eds.)
	Uncertainty in Artificial Intelligence, pp. 98--103. Morgan Kaufmann
	Publishers (1992)
	
	\bibitem{GiSa21}
	Gilio, A., Sanfilippo, G.: Compound conditionals, {F}réchet-{H}oeffding
	bounds, and {F}rank t-norms. International Journal of Approximate Reasoning
	\textbf{136},  168--200 (2021).
	\doi{https://doi.org/10.1016/j.ijar.2021.06.006}
	
	\bibitem{GiSa21E}
	Gilio, A., Sanfilippo, G.: Iterated conditionals and characterization of
	p-entailment. In: Vejnarov\'a, J., Wilson, N. (eds.) Symbolic and
	Quantitative Approaches to Reasoning with Uncertainty, ECSQARU 2021, LNCS,
	vol. 12897, pp. 629--643. Springer International Publishing (2021).
	\doi{https://doi.org/10.1007/978-3-030-86772-0\_45}
	
	\bibitem{Goodman}
	Goodman, I.R., Nguyen, H.T.: A theory of conditional information for
	probabilistic inference in intelligent systems: {II.} product space approach.
	Inf. Sci.  \textbf{76}(1-2),  13--42 (1994)
	
	\bibitem{GoNW91}
	Goodman, I.R., Nguyen, H.T., Walker, E.A.: Conditional Inference and Logic for
	Intelligent Systems: A Theory of Measure-Free Conditioning. North-Holland
	(1991), \url{www.dtic.mil/dtic/tr/fulltext/u2/a241568.pdf}
	
	\bibitem{Halpern2003}
	Halpern, J.Y.: Reasoning about Uncertainty. MIT Press, Cambridge, MA, USA
	(2003)
	
	\bibitem{Halpern2}
	Halpern, J.Y.: Actual Causality. The MIT Press (2016)
	
	\bibitem{Holz85}
	Holzer, S.: On coherence and conditional prevision. Bollettino dell'Unione
	Matematica Italiana  \textbf{4}(6),  441--460 (1985)
	
	\bibitem{Jeff91}
	Jeffrey, R.: {Matter-of-Fact Conditionals}. Proceedings of the Aristotelian
	Society, Supplementary Volume  \textbf{65},  161--183 (1991)
	
	\bibitem{Kauf09}
	Kaufmann, S.: Conditionals right and left: Probabilities for the whole family.
	Journal of Philosophical Logic  \textbf{38},  1--53 (2009).
	\doi{10.1007/s10992-008-9088-0}
	
	\bibitem{Isberner2001}
	Kern-Isberner, G.: Conditionals in Nonmonotonic Reasoning and Belief Revision
	– Considering Conditionals as Agents. Lecture Notes in Artificial
	Intelligence, Springer (2001)
	
	\bibitem{McGe89}
	McGee, V.: Conditional probabilities and compounds of conditionals.
	Philosophical Review  \textbf{98}(4),  485--541 (1989).
	\doi{http://dx.doi.org/10.2307/2185116}
	
	\bibitem{Miln97}
	Milne, P.: {Bruno de Finetti and the Logic of Conditional Events}. British
	Journal for the Philosophy of Science  \textbf{48}(2),  195--232 (1997),
	\url{http://www.jstor.org/stable/687745}
	
	\bibitem{NgWa94}
	{Nguyen}, H.T., {Walker}, E.A.: A history and introduction to the algebra of
	conditional events and probability logic. IEEE Transactions on Systems, Man,
	and Cybernetics  \textbf{24}(12),  1671--1675 (1994). \doi{10.1109/21.328924}
	
	\bibitem{PfSa17}
	Pfeifer, N., Sanfilippo, G.: Probabilistic squares and hexagons of opposition
	under coherence. International Journal of Approximate Reasoning  \textbf{88},
	282--294 (2017). \doi{10.1016/j.ijar.2017.05.014}
	
	\bibitem{Popper59}
	Popper, K.: The Logic of Scientific Discovery. Hutchinson \& Co., London (1959)
	
	\bibitem{Rega85}
	Regazzini, E.: Finitely additive conditional probabilities. Rendiconti del
	Seminario Matematico e Fisico di Milano  \textbf{55},  69--89 (1985)
	
	\bibitem{Renyi55}
	R{\'e}nyi, A.: On a new axiomatic theory of probability. Acta Mathematica
	Hungarica  \textbf{6},  285--335 (1955)
	
	\bibitem{Rigo88}
	Rigo, P.: Un teorema di estensione per probabilit{\`a} condizionate finitamente
	additive. Atti della XXXIV Riunione Scientifica S.I.S. pp. 27--34 (1988)
	
	\bibitem{Rooij}
	van Rooij, R., Schulz, K.: Conditionals, causality and conditional probability.
	Journal of Logic, Language and Information  \textbf{28}(1),  55--71 (2019).
	\doi{10.1007/s10849-018-9275-5},
	\url{https://doi.org/10.1007/s10849-018-9275-5}
	
	\bibitem{SGOP20}
	Sanfilippo, G., Gilio, A., Over, D., Pfeifer, N.: Probabilities of conditionals
	and previsions of iterated conditionals. International Journal of Approximate
	Reasoning  \textbf{121},  150 -- 173 (2020). \doi{10.1016/j.ijar.2020.03.001}
	
	\bibitem{SPOG18}
	Sanfilippo, G., Pfeifer, N., Over, D., Gilio, A.: Probabilistic inferences from
	conjoined to iterated conditionals. International Journal of Approximate
	Reasoning  \textbf{93}(Supplement C),  103 -- 118 (2018).
	\doi{10.1016/j.ijar.2017.10.027}
	
	\bibitem{Sanf12}
	Sanfilippo, G.: {From imprecise probability assessments to conditional
		probabilities with quasi additive classes of conditioning events}. In: Proc.
	of the Twenty-Eighth Conference on Uncertainty in Artificial Intelligence
	(UAI-12). pp. 736--745. AUAI Press, Corvallis, Oregon (2012)
	
	\bibitem{Scha68}
	Schay, G.: An algebra of conditional events. Journal of Mathematical Analysis
	and Applications  \textbf{24},  334--344 (1968).
	\doi{10.1016/0022-247X(68)90035-8}
	
	\bibitem{StJe94}
	Stalnaker, R., Jeffrey, R.: {Conditionals as random variables}. In: Eells, E.,
	Skyrms, B. (eds.) Probability and Conditionals: Belief Revision and Rational
	Decision, pp. 31--46. Cambridge University Press, New York, NY, USA (1994),
	\url{http://dl.acm.org/citation.cfm?id=207895.207899}
	
	\bibitem{Vantaggi2010}
	Vantaggi, B.: Incomplete preferences on conditional random quantities:
	Representability by conditional previsions. Math. Soc. Sci.  \textbf{60}(2),
	104--112 (2010). \doi{10.1016/j.mathsocsci.2010.06.002},
	\url{https://doi.org/10.1016/j.mathsocsci.2010.06.002}
	
\end{thebibliography}
%\bibliographystyle{kr}
%\bibliographystyle{splncs04}

\end{document}